\documentclass[11pt]{article}
\usepackage{algpseudocode}
\usepackage[ruled]{algorithm}
\usepackage{amsthm}
\usepackage{url}
\usepackage{longtable}
\usepackage{rotating}
\usepackage{graphicx}
\usepackage{color}
\usepackage{latexsym}
\usepackage{amsfonts}
\usepackage{amsmath}
\usepackage{psfrag}
\usepackage{mathtools}
\usepackage{caption}
\usepackage{enumerate}
\usepackage{epstopdf}
\usepackage[titletoc,toc,title]{appendix}
\usepackage[square,numbers]{natbib}
\newtheorem{proposition}{Proposition}
\newtheorem{assumption}{Assumption}

\newtheorem{theorem}{Theorem}

\newtheorem{lemma}{Lemma}

\newtheorem{definition}{Definition}
\newtheorem{remark}{Remark}
\def\ASABCP{\texttt{ASA-BCP}}
\def\NMBC{\texttt{NMBC}}
\def\ALGENCAN{\texttt{ALGENCAN}}
\def\LANCELOT{\texttt{LANCELOT~B}}
\def\R{\rm{I\!R}}

 \setbox\strutbox=\hbox{\vrule height7pt depth2pt width0pt}

\def\square{{\setbox0=\hbox{X}\hbox to \ht0{\vrule\hss\vbox to \ht0{
  \hrule width \ht0\vfil\hrule width \ht0}\vrule}}}

\makeatletter

\begin{document}
%\today
\pagestyle{empty}

\vskip 2cm \begin{center}{\huge A Two-Stage Active-Set Algorithm for Bound-Constrained Optimization}\end{center}
\par\bigskip
\centerline{\Large A. Cristofari$^\dag$, M. De Santis$^\ddag$, S. Lucidi$^\dag$, F. Rinaldi$^*$}
\par\bigskip\bigskip

\centerline{$^{\dag}$Dipartimento di Ingegneria Informatica, Automatica e Gestionale} \centerline{Sapienza Universit\`{a} di Roma} \centerline{Via Ariosto, 25, 00185 Rome, Italy}
\par\medskip
\centerline{$^{\ddag}$ Institut f\"ur Mathematik } \centerline{Alpen-Adria-Universit\"at Klagenfurt}
\centerline{Universit\"atsstr. 65-67, 9020 Klagenfurt, Austria}
\par\medskip
\centerline{ $^{*}$Dipartimento di Matematica}\centerline{Universit\`a  di Padova }  \centerline{Via Trieste, 63, 35121 Padua,
Italy}
\par\medskip
\centerline{e-mail (Cristofari): cristofari@dis.uniroma1.it}\centerline{e-mail (De Santis): marianna.desantis@aau.at}\centerline{e-mail (Lucidi): lucidi@dis.uniroma1.it} \centerline{e-mail (Rinaldi):
rinaldi@math.unipd.it}

\par\bigskip\noindent {\small \centerline{\bf Abstract}}
In this paper, we describe a two-stage method for solving optimization problems with bound constraints.
It combines the active-set estimate described in~\cite{facchinei:1995} with a modification of the non-monotone line search framework
recently proposed in~\cite{desantis:2012}.
In the first stage, the algorithm exploits a property of the active-set estimate that ensures a significant reduction
in the objective function when setting to the bounds all those variables estimated active.
In the second stage, a truncated-Newton strategy is used in the subspace of the variables estimated non-active.
In order to properly combine the two phases, a proximity check is included in the scheme.
This new tool, together  with the other theoretical features of the two stages, enables us to prove global convergence.
Furthermore, under additional standard assumptions, we can show that the algorithm converges at a superlinear rate.
Promising experimental results demonstrate the effectiveness of the proposed method.

\pagestyle{plain} \setcounter{page}{1}

\section{Introduction}\label{sec:introduction}

In this paper, we deal with nonlinear optimization problems with bound constraints.
In the literature, different approaches have been proposed for solving such problems.
Among them, we recall trust-region methods (see,~e.g.,~\cite{conn:1988,lin:1999}), interior-point methods
(see,~e.g.,~\cite{dennis:1998,heinkenschloss:1999,kanzow:2006}),
active-set methods (see,~e.g.,~\cite{bertsekas:1982,facchinei:2002,hager:2006,schwartz:1997,facchinei:1998,cheng:2012})
and second-order methods (see, e.g.,~\cite{andreani:2010,birgin:2002}).

Even though a large number of different methods is available, there is still a strong interest in developing efficient methods
to solve box-constrained problems. This is mainly due to the fact that many real-world applications can be modeled as
large-scale problems with bound constraints.
Furthermore, those methods are used as building blocks in many algorithmic frameworks for nonlinearly constrained problems (e.g., in
penalty-based approaches).

% Furthermore,
% those methods are used as building blocks in many algorithmic frameworks for nonlinearly constrained problems (e.g. in
% penalty-based approaches).
%
% Among the different approaches that have been proposed in literature for solving bound-constrained minimization problems,
% we recall trust-region methods (see~e.g.~\cite{conn:1988,lin:1999}), interior-point methods
% (see~e.g.~\cite{dennis:1998,heinkenschloss:1999,kanzow:2006}),
% active-set methods (see~e.g.~\cite{bertsekas:1982,facchinei:2002,hager:2006,schwartz:1997,facchinei:1998,cheng:2012})
% and second-order methods (see e.g.~\cite{andreani:2010,birgin:2002}).
% There is still a strong interest in developing efficient methods for dealing with those problems. This is mainly due
% to the fact that many real-world applications can be modeled as large-scale problems with bound constraints. Furthermore,
% those methods are used as building blocks in many algorithmic frameworks for nonlinearly constrained problems (e.g. in
% penalty-based approaches).

Recently, an active-set method, namely the \NMBC\ algorithm, was proposed in~\cite{desantis:2012}.
\NMBC\ algorithm has three main features: it makes use of the technique described in~\cite{facchinei:1995} to identify active constraints;
it builds up search directions by combining a truncated-Newton strategy (used in the subspace of the non-active constraints)
with a Barzilai--Borwein strategy~\cite{barzilai:1988} (used in the subspace of the active constraints); and it generates
a new iterate by means of a non-monotone line search procedure with backtracking.

Even though numerical results reported in~\cite{desantis:2012} were promising, the method has a drawback that might
affect its performance in some cases. Indeed, due to the fact that the search direction is given
by two different subvectors (the one generated by means of the truncated-Newton strategy
and the one obtained by means of the Barzilai--Borwein strategy), we might end up with a badly scaled direction.
When dealing with such a direction, finding a good starting stepsize can become pretty hard.

In this paper, we give a twofold contribution. On the one hand, we describe and analyze an important theoretical feature of the
active-set estimate proposed by Facchinei and Lucidi in~\cite{facchinei:1995}. In particular,
we prove that under suitable assumptions, a significant reduction in the objective function
can be obtained when setting to the bounds all those variables estimated active. In this way, we extend to box-constrained nonlinear problems
a similar result already proved in \cite{desantis:2016} for $\ell_1$-regularized least squares problems, and in~\cite{buchheim:2015} for
quadratic problems with non-negativity constraints.

On the other hand, thanks to the descent property of the active-set estimate, we are able to define a new algorithmic
scheme that overcomes the issues described above for the \NMBC\ algorithm. More specifically, we define a two-stage
algorithmic framework that suitably combines the active-set estimate proposed in~\cite{facchinei:1995} with
the non-monotone line search procedure described in~\cite{desantis:2012}. In the first stage of our framework, we set
the estimated active variables to the corresponding bounds. Then, in the second stage,
we generate a search direction in the subspace of the non-active variables
(employing a suitably modified truncated-Newton step) to get a new iterate.

There are three main differences between the method we propose here and the one in~\cite{desantis:2012}:
\begin{enumerate}
 \item thanks to the two stages, we can get rid of the Barzilai--Borwein
step for the active variables, thus avoiding the generation of badly scaled search directions;
\item the search direction is computed only in the subspace of the non-active variables, allowing savings in terms of CPU time,
especially when dealing with large-scale problems;
\item a specific proximity check is included in order to guarantee global convergence of the method. This is crucial, from a theoretical point of view,
      since we embed the two stages described above within a non-monotone stabilization framework.
\end{enumerate}

Regarding the theoretical properties of the algorithm, we prove that a non-monotone strategy is able to guarantee global convergence to stationary points
even if at each iteration a gradient-related direction is generated only in the subspace of the non-active variables.
Furthermore, we prove that, under standard additional assumptions, the algorithm converges at a superlinear rate.

The paper is organized as follows.
In Section~\ref{sec:problemDefinition}, we formalize the problem and introduce the notation that will be used throughout the paper.
In Section~\ref{sec:act_set}, we present our active-set estimate, stating some theoretical results, proved in Appendix~\ref{app:as}.
In Section~\ref{sec:alg}, we describe our two-stage active-set algorithm (a formal description of the algorithm can be found in Appendix~\ref{app:alg})
and report the theorems related to the convergence. The detailed convergence analysis of the algorithm is reported in Appendix~\ref{app:conv}.
Finally, our numerical experience is presented in Section~\ref{sec:numres}, and some conclusions are drawn in Section~\ref{sec:conc}.

\section{Problem Definition and Notations}\label{sec:problemDefinition}
We address the solution of bound-constrained problems of the form:
\begin{equation}\label{prob}
\min\ \{f(x): l\le x\le u\},
\end{equation}
where $f \in C^{2}(\R^n)$; $x,l,u \in \R^n$, and $l<u$.

In the following, we denote by $g(x)$ and $H(x)$ the $n$ gradient vector
and the $n\times n$ Hessian matrix of $f(x)$, respectively.
We also indicate with $\|\cdot\|$ the Euclidean norm. Given a vector $v\in \R^{n}$ and an index set $I \subseteq
\{1,\ldots,n\}$, we denote by $v_I$ the subvector with components $v_i$, $i\in I$.
Given a matrix \mbox{$H \in \mathcal \R^{n\times n}$}, we denote
by $H_{I \,I}$ the submatrix with components $h_{ij}$ with \mbox{$i,j \in I$}, and
by $\lambda_{\text{max}}$ its largest eigenvalue.
The open ball with center $x$ and radius $\rho>0$ is denoted by ${\cal B}(x,\rho)$.
Finally, given $x \in \R^n$, we indicate with $[x]^{\sharp}$ the projection of $x$ onto $[l,u]$, where $l,u \in \R^n$ define
the feasible region of problem~\eqref{prob}.

Now, we give the formal definition of stationary points for problem~\eqref{prob}.
\begin{definition}
A point $x^{*}\in [l,u]$ is called stationary point of problem~\eqref{prob}
iff it satisfies the following first-order necessary optimality conditions:
\begin{align}
\label{opt1} g_i(x^{*}) \ge 0, \quad & \text{ if  } x^{*}_{i}=l_{i}, \\
\label{opt2} g_i(x^{*}) \le 0, \quad & \text{ if  } x^{*}_{i}=u_{i}, \\
\label{opt3} g_i(x^{*}) = 0, \quad & \text{ if  } l_{i}< x^{*}_{i}< u_{i}.
\end{align}
These conditions can be equivalently written as:
\begin{align}
\label{kkt1} & g(x^*) - \lambda^* + \mu^* = 0, \\
\label{kkt2} & (l_i-x^*_i)\,\lambda^*_i = 0, \quad \qquad \,\,\,i=1,\ldots,n, \\
\label{kkt3} & (x^*_i-u_i)\,\mu^*_i = 0,  \quad \qquad \,i=1,\ldots,n,\\
\label{kkt4} & \lambda_i^* \ge 0, \quad \mu_i^* \ge 0, \quad  \qquad i=1,\ldots,n,
\end{align}
where $\lambda^*, \mu^* \in \R^{n}$ are the KKT multipliers.
\end{definition}

\section{Active-Set Estimate: Preliminary Results and Properties}\label{sec:act_set}
As we will see later on, the use of a technique to estimate the active constraints plays a crucial role in the development of
a theoretically sound and computationally efficient algorithmic framework.
The active-set estimation we consider for box-constrained nonlinear problems takes inspiration from the approach
first proposed in~\cite{dipillo:1984}, and further studied in~\cite{facchinei:1995},
based on the use of some approximations of KKT multipliers.

Let $x$ be any feasible point, and $\lambda(x)$, $\mu(x)$ be some appropriate approximations of the KKT multipliers $\lambda^*$ and $\mu^*$.
We define the following index subsets:
\begin{equation}\label{Al}
A_l(x) := \left\{i\in \{1,\ldots,n \} \, \colon \, l_i \le x_i\le l_i + \epsilon\lambda_i(x), \, g_i(x)>0 \right\},
\end{equation}
\begin{equation}\label{Au}
A_u(x) := \left\{i\in \{1,\ldots,n \} \, \colon \, u_i - \epsilon\mu_i(x) \le x_i \le u_i, \, g_i(x)<0 \right\},
\end{equation}
\begin{equation}\label{N}
N(x) := \left\{i \in \{1,\dots,n\} \, \colon \, i \notin A_l(x) \cup A_u(x)\right\},
\end{equation}
where $\epsilon > 0$.

In particular, $A_l(x)$ and $A_u(x)$ contain the indices of the variables estimated active at the lower bound and the upper bound, respectively.
The set $N(x)$ includes the indices of the variables estimated non-active.

In this paper, $\lambda(x)$ and $\mu(x)$ are defined as the multiplier functions introduced in~\cite{grippo2:1991}:
starting from the solution of~\eqref{kkt1} at $x$, and then minimizing the error over~\eqref{kkt2}--\eqref{kkt4}, it is possible to compute
the functions $\lambda\colon\R^n\rightarrow\R^n$ and $\mu\colon\R^n\rightarrow\R^n$ as:
\begin{align}
& \label{lambdafunc}\lambda_i(x) := \frac{(u_i-x_i)^2}{(l_i-x_i)^2+(u_i-x_i)^2} g_i(x), \qquad i = 1,\dots,n, \\
& \label{mufunc}\mu_i(x) := -\frac{(l_i-x_i)^2}{(l_i-x_i)^2+(u_i-x_i)^2} g_i(x), \, \quad i = 1,\dots,n.
\end{align}

By adapting the results shown in~\cite{facchinei:1995}, we can state the following proposition.
\begin{proposition}\label{prop:estim}
If $(x^*,\lambda^*,\mu^*)$ satisfies KKT conditions for problem~\eqref{prob}, then there exists a neighborhood ${\cal B}(x^*,\rho)$ such that
\begin{itemize}
\item[] $\{i \, \colon \, x^*_i = l_i, \, \lambda^*_i > 0 \} \subseteq A_l(x) \subseteq \{i \, \colon \, x^*_i = l_i\}$,
\item[] $\{i \, \colon \, x^*_i = u_i, \, \mu^*_i > 0 \} \subseteq A_u(x) \subseteq \{i \, \colon \, x^*_i = u_i\}$,
\end{itemize}
for each $x \in {\cal B}(x^*,\rho)$.

Furthermore, if strict complementarity holds, then
\begin{itemize}
\item[] $\{i \, \colon \, x^*_i = l_i, \, \lambda^*_i > 0 \} = A_l(x) = \{i \, \colon \, x^*_i = l_i\}$,
\item[] $\{i \, \colon \, x^*_i = u_i, \, \mu^*_i > 0 \} = A_u(x) = \{i \, \colon \, x^*_i = u_i\}$,
\end{itemize}
for each $x \in {\cal B}(x^*,\rho)$.
\end{proposition}

We notice that stationary points can be characterized by using the active-set estimate, as shown in the next propositions.
\begin{proposition}\label{prop:stationary_point}
A point $\bar{x} \in [l,u]$ is a stationary point of problem~\eqref{prob} iff the following conditions hold:
\begin{gather}
\label{opt_cond_1} \max\ \{l_i - \bar x_i, -g_i(\bar x)\} = 0, \quad i \in A_l(\bar x), \\
\label{opt_cond_2} \max\ \{\bar x_i - u_i, g_i(\bar x)\} = 0, \quad i \in A_u(\bar x), \\
\label{opt_cond_3} g_i(\bar x) = 0, \quad i \in N(\bar x).
\end{gather}
\end{proposition}

\begin{proof}
See Appendix~\ref{app:as}.
\end{proof}

\begin{proposition}\label{prop:stationary_point2}
Given $\bar{x} \in [l,u]$, assume that
\begin{equation}\label{active_set_empty}
\{i \in A_l(\bar x) \, \colon \, \bar x_i>l_i\} \cup \{i \in A_u(\bar x) \, \colon \, \bar x_i < u_i\} = \emptyset.
\end{equation}
Then, $\bar x$ is a stationary point of problem~\eqref{prob} iff
\[
g_i(\bar x) = 0 \text{ for all } i \in N(\bar x).
\]
\end{proposition}

\begin{proof}
See Appendix~\ref{app:as}.
\end{proof}

\begin{proposition}\label{prop:stationary_point3}
Given $\bar{x} \in [l,u]$, assume that
\begin{equation}\label{stationary_free}
g_i(\bar x) = 0 \text{ for all } i \in N(\bar x).
\end{equation}
Then, $\bar x$ is a stationary point of problem~\eqref{prob} iff
\[
\{i \in A_l(\bar x) \, \colon \, \bar x_i > l_i\} \cup \{i \in A_u(\bar x) \, \colon \, \bar x_i < u_i\} = \emptyset.
\]
\end{proposition}

\begin{proof}
See Appendix~\ref{app:as}.
\end{proof}

\subsection{Descent Property of the Active-Set}
In this subsection, we show that the active-set estimate can be used for computing a point that ensures a sufficient decrease in the objective function simply
by fixing the estimated active variables at the bounds.

First, we give an assumption on the parameter $\epsilon$ appearing in the definition of the active-set estimates $A_l(x)$ and $A_u(x)$ that will be used to prove the
main result in this subsection.

\begin{assumption}\label{ass:eps}
Assume that the parameter $\epsilon$ appearing in~\eqref{Al} and~\eqref{Au} satisfies the following conditions:
\begin{equation}\label{eq_ass:eps}
\left \{
\begin{array}{ll}
0 < \epsilon \le \dfrac{1}{\bar{\lambda}}, & \quad \mbox{ if } \ \bar{\lambda} > 0, \\
\epsilon > 0, & \quad \mbox{ otherwise, }
\end{array}
\right .
\end{equation}
where
\[
\bar{\lambda} := \max_{x \in [l,u]} \lambda_{\text{max}} (H(x)).
\]
\end{assumption}

Now, we state the main result of the subsection.

\begin{proposition}\label{prop1}
Let Assumption~\ref{ass:eps} hold. Let $x\in [l,u]$ be such that
\[
A_l(x) \cup A_u(x) \ne \emptyset,
\]
and let $\tilde x$ be the point defined as
\begin{align*}
\tilde x_i := l_i, \quad & i \in A_l(x), \\
\tilde x_i := u_i, \quad & i \in A_u(x), \\
\tilde x_i := x_i, \quad & i \in N(x),
\end{align*}
where $A_l(x)$, $A_u(x)$ and $N(x)$ are the index subsets defined as in~\eqref{Al}, \eqref{Au} and~\eqref{N}, respectively.

Then,
\[
f(\tilde x)-f(x) \le - \dfrac{1}{2\epsilon}\|x-\tilde x\|^2.
\]
\end{proposition}

\begin{proof}
See Appendix \ref{app:as}.
\end{proof}

As we already highlighted in the Introduction, Proposition \ref{prop1} is a non-trivial extension of similar results already proved in the literature.

In particular, here we deal with problems having a general non-convex objective function, while in \cite{desantis:2016,buchheim:2015}, where a similar analysis was carried out,
the authors only considered convex quadratic optimization problems.

\subsection{Descent Property of the Non-active Set}
In this subsection, we show that, thanks to the theoretical properties of the active-set estimate, a sufficient decrease in the objective function can also be obtained
by suitably choosing a direction in the subspace of the non-active variables only. Let us consider a search direction satisfying the following conditions:
\begin{gather}
d_i = 0, \quad \forall i\in A_l(x) \cup A_u(x), \label{conddir1} \\
d_{N(x)}^Tg_{N(x)}(x) \le -\sigma_1 \|g_{N(x)}(x)\|^{2}, \label{conddir2} \\
\|d_{N(x)}\| \le \sigma_{2}\|g_{N(x)}(x)\|, \label{conddir3}
\end{gather}
where $\sigma_1, \sigma_2 > 0$. Condition \eqref{conddir1} ensures that the estimated active variables are not updated when moving along such a direction, while~\eqref{conddir2} and~\eqref{conddir3} imply that
$d$ is gradient-related with respect to only the estimated non-active variables.

Given a direction $d$ satisfying~\eqref{conddir1}--\eqref{conddir3}, the following proposition shows that a sufficient decrease in the objective function
can be guaranteed by projecting suitable points obtained along $d$.

\begin{proposition}\label{prop:line_search}
Given $\bar x \in [l,u]$, let us assume that $N(\bar x) \ne \emptyset$ and that $g_{N(\bar x)}(\bar x) \ne 0$. Let $\gamma \in ]0,1[$. Then, there exists $\bar{\alpha}>0$ such that
\begin{equation}\label{prop_decr}
f(\bar x(\alpha))-f(\bar x)\le \gamma \alpha g(\bar x)^T d, \quad \forall \alpha\in ]0,\bar{\alpha}],
\end{equation}
where $\bar x(\alpha):=[\bar x+\alpha d]^{\sharp}$, and $d$ satisfies~\eqref{conddir1}--\eqref{conddir3} in $\bar x$.
\end{proposition}

\begin{proof}
See Appendix~\ref{app:as}.
\end{proof}

\section{A New Active-Set Algorithm for Box-Constrained Problems}\label{sec:alg}
In this section, we describe a new algorithmic framework for box-constrained problems. Its distinguishing
feature is the presence of two different stages that enable us to separately handle active and non-active variables.

In Appendix~\ref{app:alg}, we report the formal scheme of our Active-Set Algorithm for Box-Constrained Problems (\ASABCP).
In the following, we only give a sketch of it, indicating with $f_R$
a reference value of the objective function that is updated throughout the procedure.
Different criteria were proposed in the literature to choose this value (see, e.g.,~\cite{zhang:2004}).
Here, we take $f_R$ as the maximum among the last $M$ function evaluations, where $M$ is a nonnegative parameter.

\begin{itemize}
\item At every iteration $k$, starting from the non-stationary point $x^k$, the algorithm fixes the
      estimated active variables at the corresponding bounds, thus producing the new point $\tilde x^k$.
      In particular, the sets
      \begin{equation}\label{active_sets}
      A_l^k:=A_l(x^k), \quad A_u^k:=A_u(x^k)\quad\mbox{ and } \quad N^k:=N(x^k)
      \end{equation}
      are computed and the point $\tilde x^k$ is produced by setting
      \[
      \tilde x^{k}_{A_l^k} := l_{A_l^k},\quad \tilde x^{k}_{A_u^k} := u_{A_u^k}\quad\mbox{ and }\quad\tilde x^{k}_{N^k} := x^k_{N^k}.
      \]
\item Afterward, a check is executed to verify if the new point $\tilde x^k$ is sufficiently close to $x^k$.
      If this is the case, the point $\tilde x^k$ is accepted.
      Otherwise, an objective function check is executed and two further cases are possible:
      if the objective function is lower than the reference value $f_R$, then we accept the point $\tilde x^k$;
      otherwise the algorithm sets $\tilde x^k$ by backtracking to the last good point (i.e., the point $\tilde x$ that produced the last $f_R$).
\item At this point, the active and non-active sets are updated considering the information related to $\tilde x^k$, i.e., we build
      \begin{equation}\label{active_sets_x_tilde}
      \tilde A_l^k:=A_l(\tilde x^k), \quad \tilde A_u^k:=A_u(\tilde x^k)\quad\mbox{ and } \quad \tilde N^k:=N(\tilde x^k).
      \end{equation}
      A search direction $d^k$ is then computed: we set $d_{\tilde A^k}^k :=0 $, with \mbox{$\tilde A^k := \tilde A_l^k \cup \tilde A_u^k$,}
      and calculate $d^{k}_{\tilde N^k}$ by means of a modified truncated-Newton step
      (see, e.g.,~\cite{dembo:1983} for further details on truncated-Newton approaches).
\item Once $d^k$ is computed, a non-monotone stabilization strategy, inspired by the one proposed in~\cite{grippo:1991},
      is used to generate the new iterate.
      In particular, the algorithm first checks if $\|d^k\|$ is sufficiently small. If this is the case, the unitary stepsize
      is accepted, and we set
      \[
      x^{k+1} := [\tilde x^k + d^k]^{\sharp}
      \]
      without computing the related objective function value and start a new iteration.\\
      Otherwise, an objective function check is executed and two further cases are possible:
      if the objective function is greater than or equal to the reference value $f_R$,
      then we backtrack to the last good point and take the related search direction; otherwise we continue with the current point.
      Finally, a non-monotone line search is performed in order to get a stepsize $\alpha^k$ and generate
      \[
      x^{k+1} := [\tilde x^k+\alpha^k d^k]^{\sharp}.
      \]
\item After a prefixed number of iterations without calculating the objective function, a check is executed to verify if the objective
      function is lower than the reference value $f_R$. If this is not the case, a backtracking and a non-monotone line search are executed.
\end{itemize}

The \emph{non-monotone line search} used in the algorithm is the same as the one described in, e.g., \cite{desantis:2012}.
It sets $\alpha^k := \delta^{\nu}$, where $\nu$ is the smallest nonnegative integer for which
\begin{equation}\label{suffdec}
f([\tilde x^k+\delta^{\nu} d^k]^{\sharp}) \le  f_R + \gamma \delta^{\nu} g(\tilde x^k)^T d^k,
\end{equation}
with $\delta\in ]0,1[$ and $\gamma\in ]0,\frac 1 2[$.

\begin{remark}
From Proposition~\ref{prop:stationary_point2} and~\ref{prop:stationary_point3},
it follows that \ASABCP\ is well defined, in the sense that at the $k$-th iteration it produces a new point $x^{k+1} \ne x^k$ iif $x^k$ is non-stationary.
\end{remark}

Hereinafter, we indicate the active-set estimates in $x^k$ and $\tilde x^k$ with the notation used in~\eqref{active_sets} and in~\eqref{active_sets_x_tilde}, respectively.

Now, we state the main theoretical result ensuring the global convergence of \ASABCP.
\begin{theorem}\label{th:glob_conv}
Let Assumption~\ref{ass:eps} hold. Then, \ASABCP\ either produces a stationary point for problem~\eqref{prob} in a finite number of iterations,
or produces an infinite sequence $\{x^k\}$ and every limit point $x^*$ of the sequence is a stationary point for problem~\eqref{prob}.
\end{theorem}

\begin{proof}
See Appendix~\ref{app:conv}.
\end{proof}

Finally, under standard additional assumptions, superlinear convergence of the method can be proved.
\begin{theorem}\label{th:superlin}
Assume that $\{x^k\}$ is a sequence generated by \ASABCP\ converging to a point $x^*$ satisfying the strict complementarity condition
and such that $H_{N^* N^*} (x^*) \succ 0$, where $N^* := \{i \colon l_i < x^*_i < u_i\}$.
Assume that the sequence $\{d^k\}$ of directions satisfies the following condition:
\begin{equation}\label{tronc}
\lim_{k \to \infty} \frac{\| H_{\tilde N^k \tilde N^k}(\tilde x^k) d^k_{\tilde N^k} + g_{\tilde N^k}(\tilde x^k) \|}{\|g_{\tilde N^k}(\tilde x^k) \|} = 0.
\end{equation}
Then, the sequence $\{x^k\}$ converges to $x^*$ superlinearly.
\end{theorem}

\begin{proof}
See Appendix~\ref{app:conv}.
\end{proof}

\section{Numerical Experience}\label{sec:numres}
In this section, we describe the details of our computational experience.

%In Subsection~\ref{impissues}, we describe the implementation details related to our algorithmic scheme.
In Subsection~\ref{compOther}, we compare \ASABCP\ with the following codes:
\begin{itemize}
\item \NMBC~\cite{desantis:2012} (in particular, we considered the version named \NMBC$_2$ in~\cite{desantis:2012});
\item \ALGENCAN~\cite{birgin:2002}: an active-set method using spectral projected gradient steps for leaving faces,
downloaded from the TANGO web page (\url{http://www.ime.usp.br/~egbirgin/tango});
\item \LANCELOT~\cite{gould:2003}: a Newton method based on a trust-region strategy, downloaded from the GALAHAD web page (\url{http://www.galahad.rl.ac.uk}).
\end{itemize}

All computations have been run on an Intel Xeon(R), CPU E5-1650 v2 3.50 GHz.
The test set consisted of $140$ bound-constrained problems from the CUTEst collection~\cite{gould:2015}, with dimension up to $10^5$.
The stopping condition for all codes was
\[
\|x - [x - g(x)]^\sharp\|_\infty< 10^{-5},
\]
where $\|\cdot\|_\infty$ denotes the sup-norm of a vector.

In order to compare the performances of the algorithms, we make use of the performance profiles proposed in~\cite{dolan:2002}.

Following the analysis suggested in~\cite{birgin:2012}, we preliminarily checked whether
the codes find different stationary points: the comparison is
thus restricted to problems for which all codes find the same stationary point (with a tolerance of $10^{-3}$).
Furthermore, we do not consider in the analysis those problems for which all methods find a stationary point in less
than $1$ second.

In \ASABCP, we set the algorithm parameters to the following values: \mbox{$Z := 20$} and \mbox{$M := 99$}
(so that the last $100$ objective function values are included in the computation of the reference value).

In running the other methods considered in the comparisons, default values were used for all parameters
(but those related to the stopping condition).

% More specifically:
% \begin{itemize}
% \item In \NMBC, a non-monotone strategy is employed (it is similar to the one described for \ASABCP), where the parameters $Z$ and $M$ are equal to $20$ and $100$, respectively.
%     Moreover, the parameter $\eps$ used in the active-set estimate is equal to $10^{-4}$.
% \item In \LANCELOT, a band preconditioner is employed for the conjugate gradient method, with a semi-bandwidth equal to $5$.
%     Moreover, a non-monotone strategy is used with a history-length equal to $1$.
% \item In \ALGENCAN, the truncated-Newton method is used as inner solver method, the scaling feature is disabled, and
%     the parameter $\eta$ is equal to $0.1$ (a face of the feasible set is abandoned by the algorithm when the norm of
%     the internal components of the continuous projected gradient is smaller than $\eta$ times the norm of the continuous
%     projected gradient).
% \end{itemize}

C++ and Fortran 90 implementations (with CUTEst interface) of \ASABCP, together with details related to the experiments and the implementation,
can be found at the following web page:
\url{https://sites.google.com/a/dis.uniroma1.it/asa-bcp}.

\subsection{Comparison on CUTEst Problems}\label{compOther}
In this subsection, we first compare \ASABCP\ with the \NMBC\ algorithm presented in~\cite{desantis:2012}.
Then, we report the comparison of \ASABCP\ with other two solvers for bound-constrained problems,
namely \ALGENCAN~\cite{birgin:2002} and \LANCELOT~\cite{gould:2003}.
All the codes are implemented in Fortran 90.

Recalling how we selected the relevant test problems,
the analysis was restricted to $43$ problems for the comparison between \ASABCP\ and \NMBC,
and to $62$ problems for the comparison between \ASABCP, \ALGENCAN\ and \LANCELOT.

In particular, in the comparison between \ASABCP\ and \NMBC, $96$ problems were discarded
because they were solved in less than $1$ second by both algorithms. A further problem (namely \textit{SCOND1LS} with $5002$ variables)
was removed because \ASABCP\ and \NMBC\ found two different stationary
points (\NMBC\ found the worst one).

In the comparison between \ASABCP, \ALGENCAN\ and \LANCELOT, $75$ problems were discarded
because they were solved in less than $1$ second by all the considered algorithms.
Other $3$ problems were removed as the methods stopped at different
stationary points. Namely, \textit{NCVXBQP3} with $10^5$ variables,
\textit{POWELLBC} with $10^3$ variables and \textit{SCOND1LS} with
$5002$ variables were discarded in our comparison. The worst stationary points were found by \ASABCP,
\LANCELOT\ and \ASABCP, respectively.

In Figure~\ref{fig:perf_plots_ASABCP_NMBC}, we report the performance profiles of \ASABCP\ and \NMBC.
These profiles show that \ASABCP\ outperforms \NMBC\ in terms of CPU time,
number of objective function evaluations and number of conjugate gradient iterations.
This confirms the effectiveness of our two-stage approach when compared to the \NMBC\ algorithm.

\begin{figure}
\centering
\includegraphics[scale=0.6, trim=1.2cm 0cm 1.2cm 0]{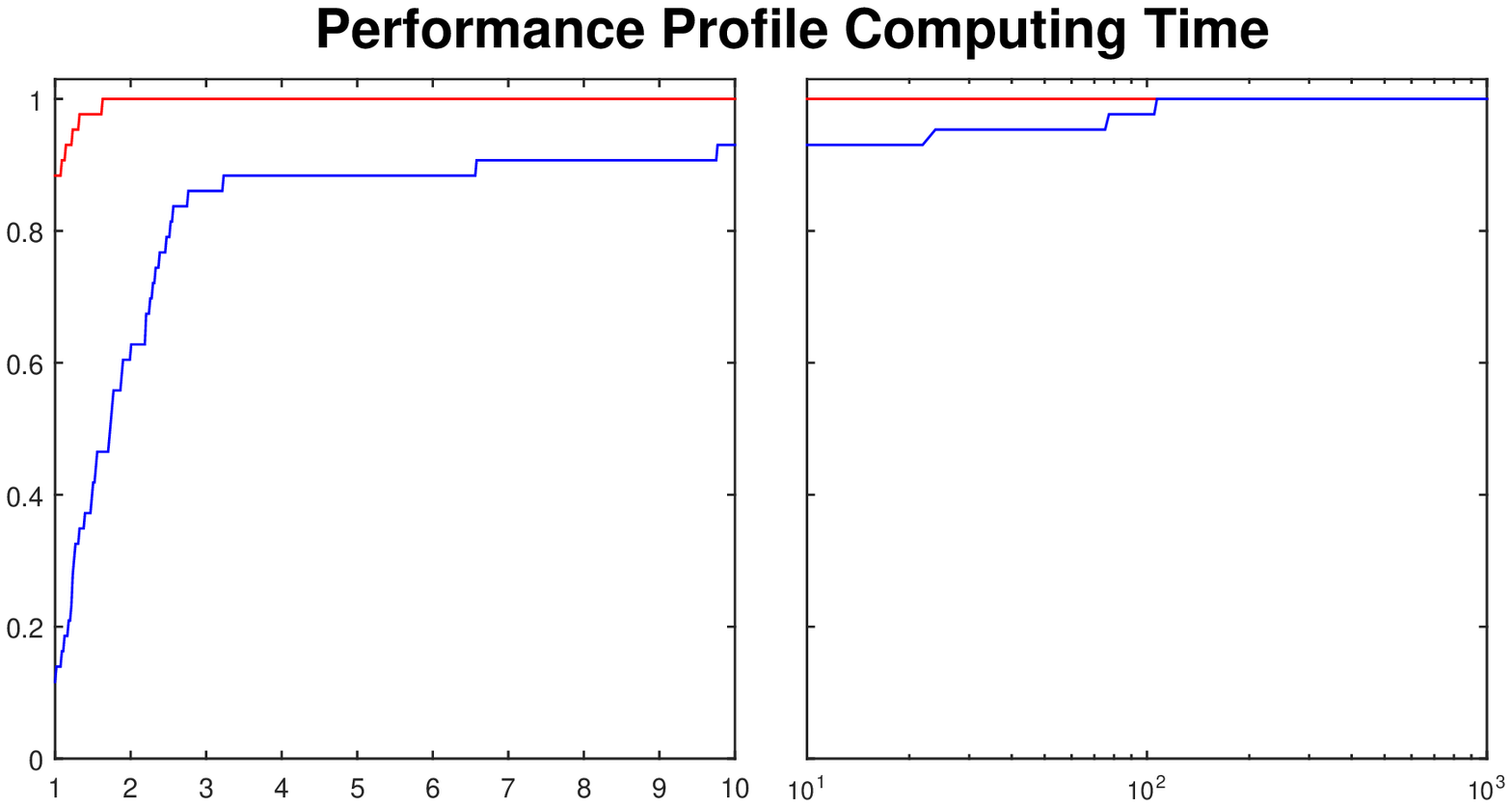}
\includegraphics[scale=0.6, trim=1.2cm 0cm 1.2cm 0]{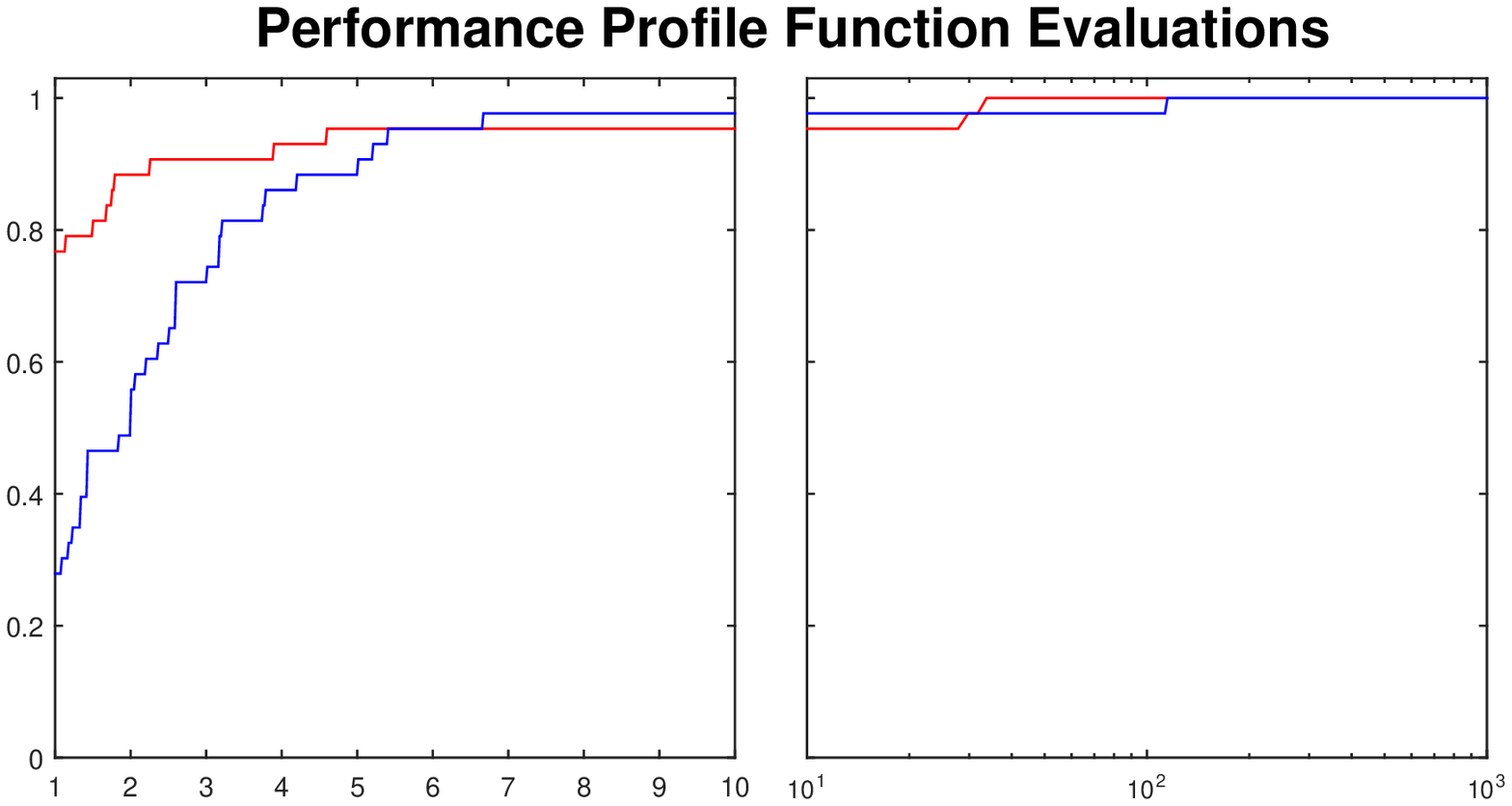}
\includegraphics[scale=0.6, trim=1.2cm 0cm 1.2cm 0]{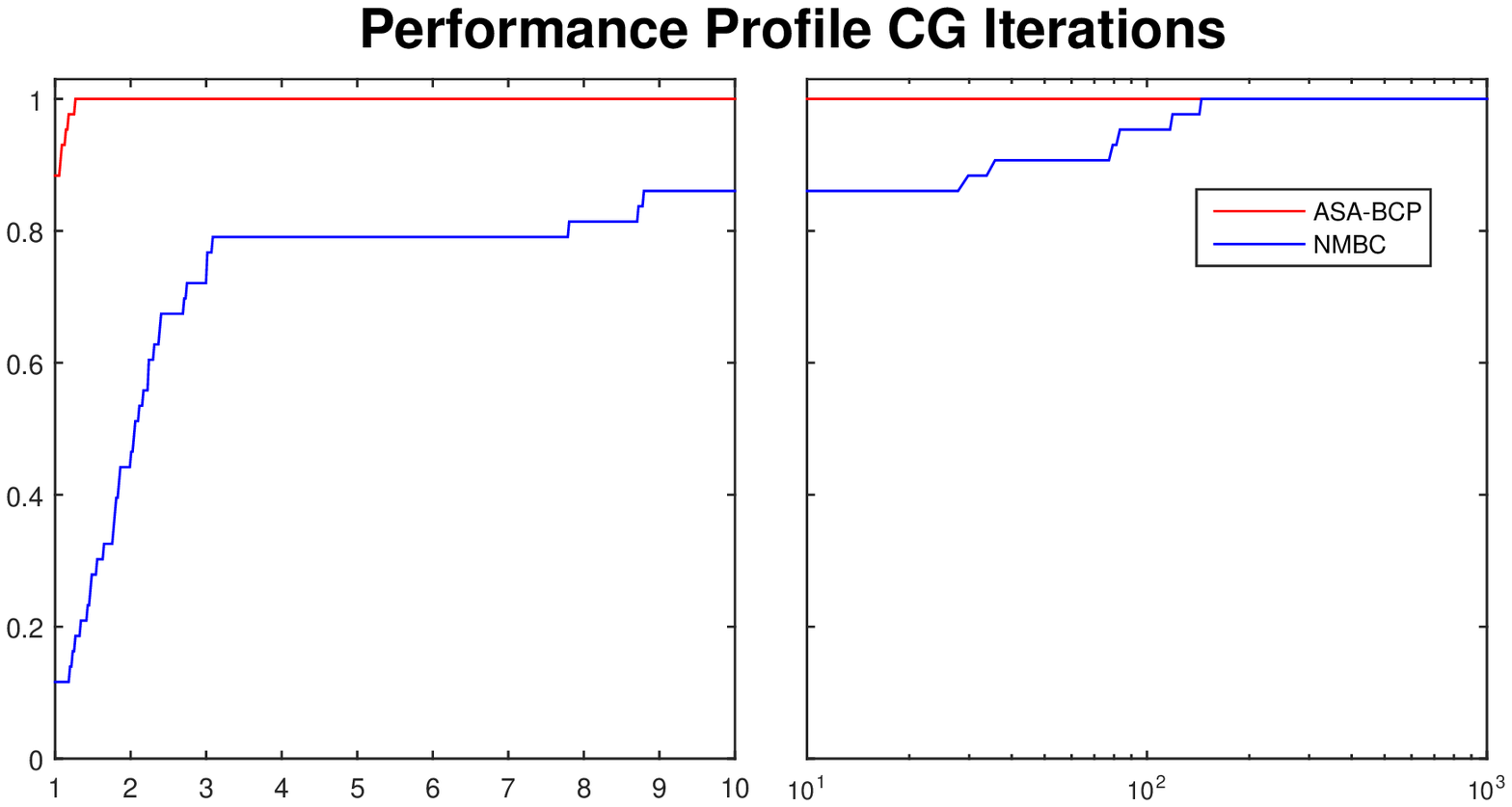}
\caption{Comparison between \ASABCP\ and \NMBC: performance profiles on CPU time,
number of objective function evaluations and number of conjugate gradient iterations.
The $x$ axis is in linear scale in the left panel and in logarithmic scale
in the right panel.}
\label{fig:perf_plots_ASABCP_NMBC}
\end{figure}

\begin{figure}
\centering
\includegraphics[scale=0.6, trim=1.2cm 0cm 1.2cm 0]{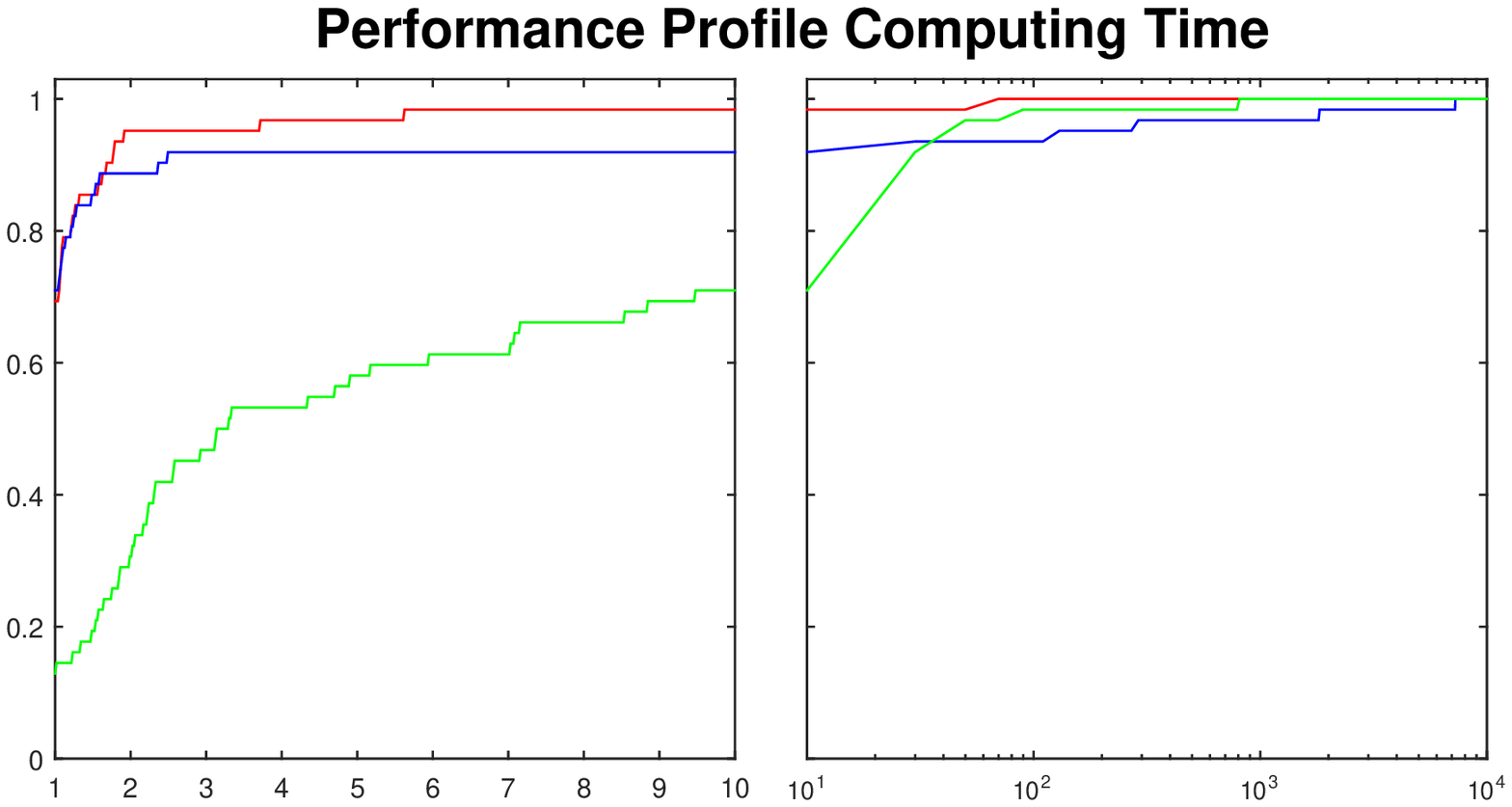}
\includegraphics[scale=0.6, trim=1.2cm 0cm 1.2cm 0]{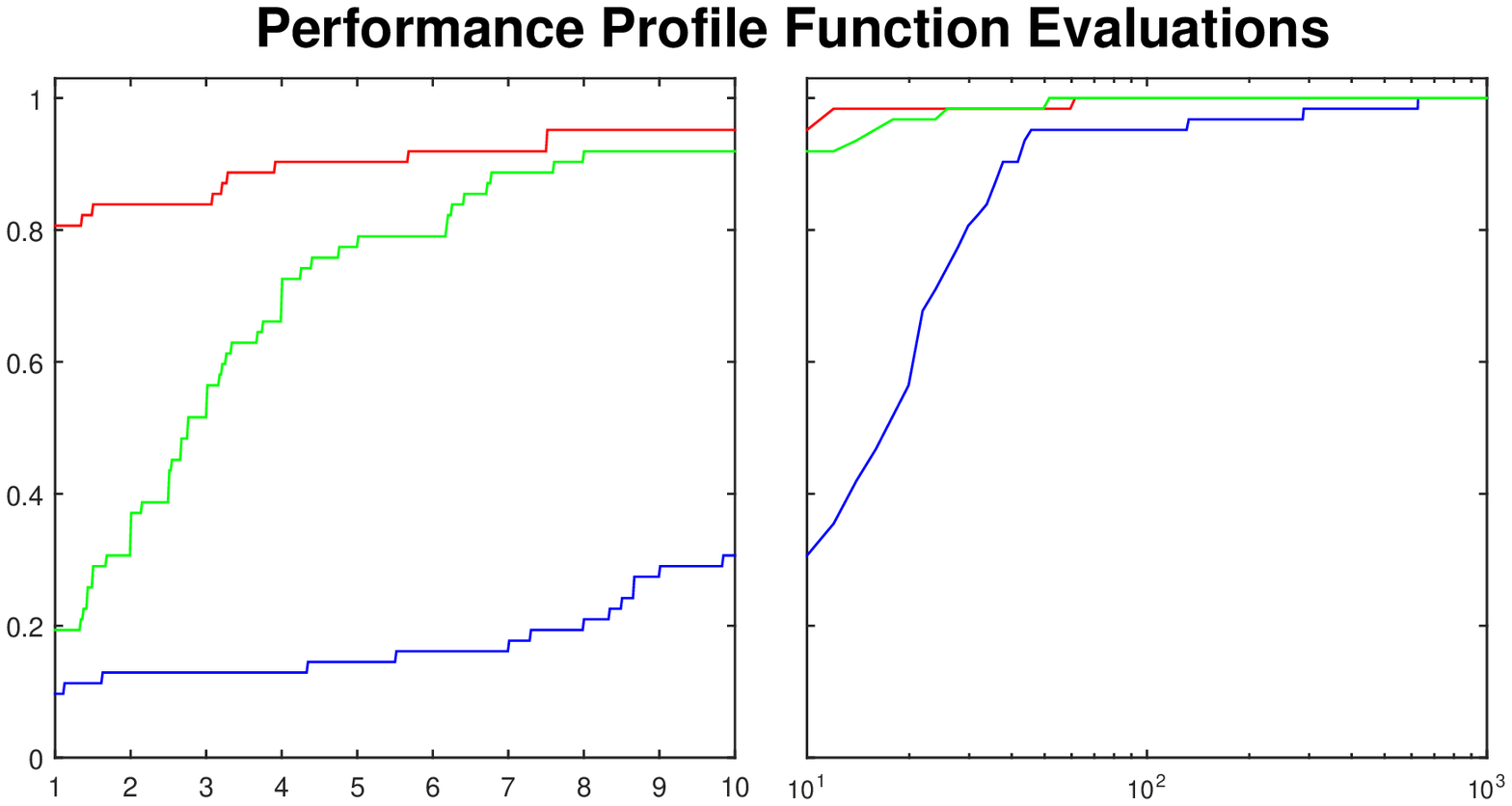}
\includegraphics[scale=0.6, trim=1.2cm 0cm 1.2cm 0]{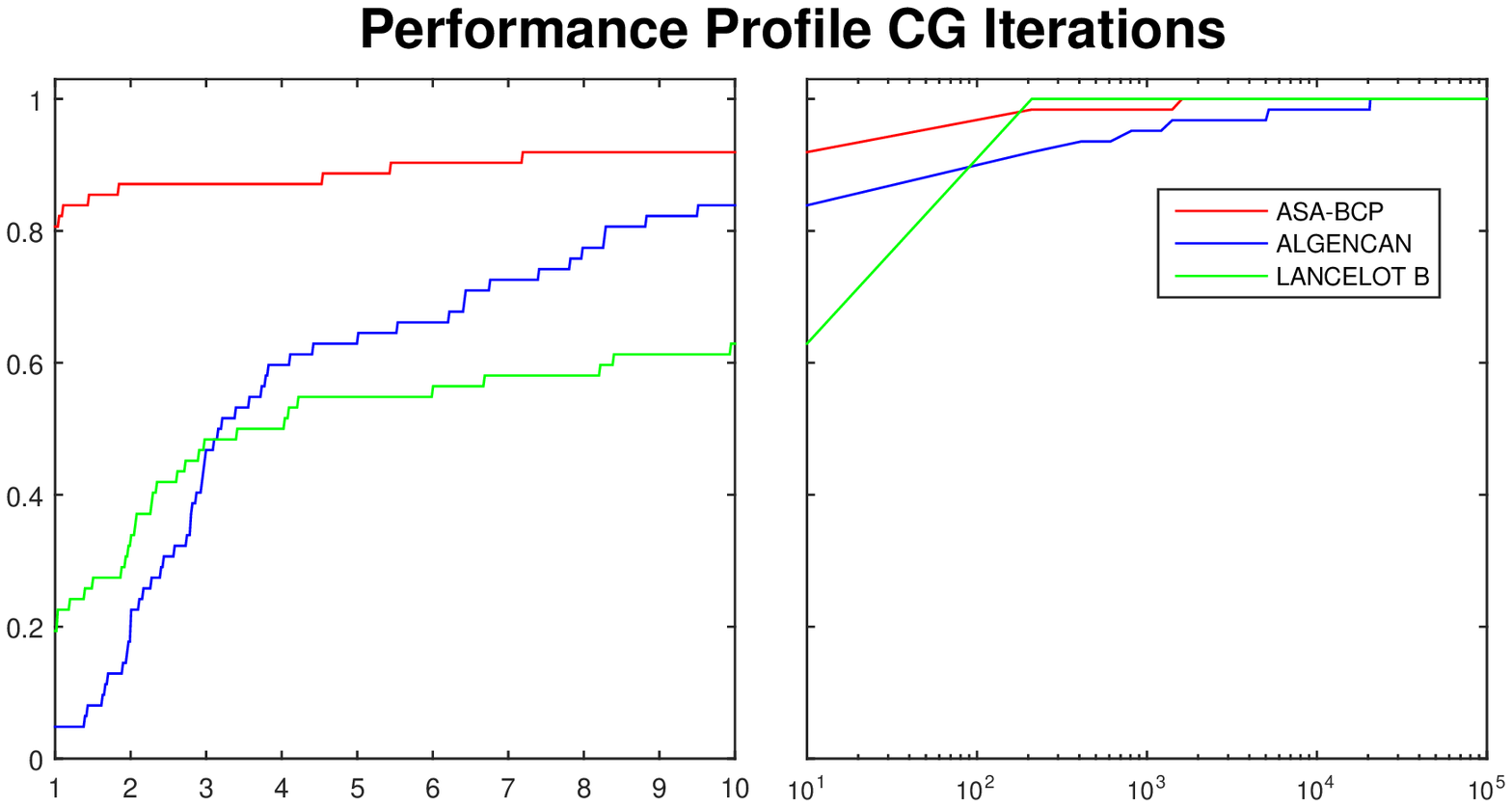}
\caption{Comparison among \ASABCP, \ALGENCAN\ and \LANCELOT: performance profiles on CPU time,
number of objective function evaluations and number of conjugate gradient iterations.
The $x$ axis is in linear scale in the left panel and in logarithmic scale in the right panel.}
\label{fig:perf_plots_ASABCP_LANCELOT_ALGENCAN}
\end{figure}

These results seem to confirm that on the one hand, computing the search direction only in the subspace of the non-active variables
guarantees some savings in terms of CPU time, and, on the other hand, getting rid of the Barzilai--Borwein step (used in \NMBC)
avoids the generation of badly scaled search directions.

In Figure~\ref{fig:perf_plots_ASABCP_LANCELOT_ALGENCAN}, we show the performance profiles of \ASABCP, \ALGENCAN\ and \LANCELOT.
By taking a look at the performance profiles related to CPU time, we can easily see that \ASABCP\ and \ALGENCAN\ are comparable in terms of efficiency
and are both better than \LANCELOT. As regards robustness, we can see that \ASABCP\ outperforms both \ALGENCAN\ and \LANCELOT.
More specifically, when $\tau$ is equal to $2$, \ASABCP\ solves about $95\%$ of the problems, while \ALGENCAN\ and \LANCELOT\ respectively, solve
about $90\%$ and $30\%$ of the problems.
Furthermore, \ASABCP\ is able to solve all the problems when $\tau$ is about $70$, while \ALGENCAN\ and \LANCELOT\ get to solve all the problems
for significantly larger values of $\tau$.

For what concerns the number of objective function evaluations, \ASABCP\ is the best in terms of efficiency
and is competitive with \LANCELOT\ in terms of robustness.
In particular, when $\tau$ is equal to $2$, \ASABCP\ solves about $85\%$ of the problems, while \ALGENCAN\ and \LANCELOT\ respectively, solve
about $15\%$ and $30\%$ of the problems.
Moreover, \ASABCP\ and \LANCELOT\ solve all the problem when $\tau$ is about $60$ and $50$, respectively,
while \ALGENCAN\ gets to solve all the problems when $\tau$ is about $600$.

Finally, as regards the number of conjugate gradient iterations, \ASABCP\ outperforms the other two codes
in terms of efficiency, while \LANCELOT\ is the best in terms of robustness.
More in detail, when $\tau$ is equal to $2$, \ASABCP\ solves about $85\%$ of the problems, while \ALGENCAN\ and \LANCELOT\ respectively, solve
about $20\%$ and $35\%$ of the problems. \LANCELOT\ is able to solve all the problems when $\tau$ is about 200,
while \ASABCP\ and \ALGENCAN\ need larger values of~$\tau$.

\section{Conclusions}\label{sec:conc}
In this paper, a two-stage active-set algorithm for box-constrained nonlinear programming problems is devised.
In the first stage, we get a significant reduction in the objective function simply by setting to the bounds
the estimated active variables.
In the second stage, we employ a truncated-Newton direction computed in the subspace of the estimated non-active variables.
These two stages are inserted in a non-monotone framework and the convergence of the resulting algorithm \ASABCP\ is proved.
Experimental results show that our implementation of \ASABCP\ is competitive with other widely used codes for
bound-constrained minimization problems.

%%%%%%%%%%%%%%%%%%%%%%%%%%%%%%%%%%%%%%%%%%%%%%%%%%%%%%%%%%%%%%%%%%%%%%%%%%%%%%%%%%%%%%%%%%%%%
%                                          APPENDIX
%%%%%%%%%%%%%%%%%%%%%%%%%%%%%%%%%%%%%%%%%%%%%%%%%%%%%%%%%%%%%%%%%%%%%%%%%%%%%%%%%%%%%%%%%%%%%

\section*{Appendices}

\begin{appendices}

\section{}\label{app:as}
\small \linespread{1}

%%%%%%%%%%%%%%%%%%%%%%%%%%%%%%%%%%%%%%%%%
% characterization of stationary points
%%%%%%%%%%%%%%%%%%%%%%%%%%%%%%%%%%%%%%%%%
\begin{proof}[Proof of Proposition~\ref{prop:stationary_point}]
Assume that $\bar{x}$ satisfies~\eqref{opt_cond_1}--\eqref{opt_cond_3}. First, we show that
\begin{gather}
\bar{x}_i = l_i, \quad \text{if } i \in A_l(\bar{x}), \label{prop:Al} \\
\bar{x}_i = u_i, \quad \text{if } i \in A_u(\bar{x}). \label{prop:Au}
\end{gather}
In order to prove~\eqref{prop:Al}, assume by contradiction that there exists an index $i \in A_l(\bar{x})$ such that
$l_i < \bar{x} \le l_i + \epsilon\lambda_i(\bar{x})$. It follows that $\lambda_i(\bar{x}) > 0$, and, from~\eqref{lambdafunc}, that $g_i(\bar x)>0$,
contradicting~\eqref{opt_cond_1}. Then,~\eqref{prop:Al} holds. The same reasoning applies to prove~\eqref{prop:Au}.\\
Recalling~\eqref{Al}, we have that $g_i(\bar x)>0$ for all $i \in A_l(\bar x)$. Combined with~\eqref{prop:Al}, it means that
$\bar x_i$ satisfies~\eqref{opt1} for all $i \in A_l(\bar x)$.
Similarly, since $g_i(\bar x)<0$ for all $i \in A_u(\bar x)$ and~\eqref{prop:Au} holds, then
$\bar x_i$ satisfies~\eqref{opt2} for all $i \in A_u(\bar x)$.\\
From~\eqref{opt_cond_3}, we also have that $\bar x_i$ satisfies
optimality conditions for all $i \in N(\bar x)$. Then, $\bar{x}$ is a stationary point.\\
Now, assume that $\bar{x}$ is a stationary point. First, we consider a generic index $i$ such that $\bar x_i = l_i$. For such an index,
from~\eqref{opt1} we get $g_i(\bar x)\ge0$. If $g_i(\bar x)>0$, then, from~\eqref{Al}, it follows that
$i \in A_l(\bar x)$ and~\eqref{opt_cond_1} is satisfied. Vice versa, if $g_i(\bar x)=0$, then we have that $i$ belongs to $N(\bar x)$, satisfying~\eqref{opt_cond_3}.
The same reasoning applies for a generic index $i$ such that $\bar x_i = u_i$.\\
Finally, for every index $i$ such that $l_i < \bar x_i < u_i$, from~\eqref{opt3} we have that $g_i(\bar x) = 0$. Then, $\bar x$ satisfies~\eqref{opt_cond_1}--\eqref{opt_cond_3}. %This completes the proof.
\end{proof}

\begin{proof}[Proof of Proposition~\ref{prop:stationary_point2}]
Assume that~\eqref{active_set_empty} is verified. Namely,
\begin{gather*}
\bar x_i = l_i, \quad \text{if } i \in A_l(\bar x), \\
\bar x_i = u_i, \quad \text{if } i \in A_u(\bar x).
\end{gather*}
Recalling the definition of $A_l(\bar x)$ and $A_u(\bar x)$, the previous relations imply that~\eqref{opt_cond_1} and~\eqref{opt_cond_2} are verified.
Then, from Proposition~\ref{prop:stationary_point}, $\bar x$ is a stationary point if and only if $g_i(\bar x) = 0$ for all $i\in N(\bar x)$. %This completes the proof.
\end{proof}

\begin{proof}[Proof of Proposition~\ref{prop:stationary_point3}]
Assume that condition~\eqref{stationary_free} is verified. If we have
\[\{i \in A_l(\bar{x})~\,~\colon \, \bar{x}_i>l_i\}~\cup~\{i \in A_u(\bar{x})~\,\colon \, \bar{x}_i<u_i\} = \emptyset,\]
from the definition of $A_l(\bar x)$ and $A_u(\bar x)$ it follows that
\begin{align*}
\bar x_i = l_i \, \text { and } \, g_i(\bar x) > 0, \quad & i \in A_l(\bar x), \\
\bar x_i = u_i \, \text { and } \, g_i(\bar x) < 0, \quad & i \in A_u(\bar x).
\end{align*}
Then, conditions~\eqref{opt1}--\eqref{opt3} are verified, and $\bar x$ is a stationary point.\\
Conversely, if $\bar x$ is a stationary point, we proceed by contradiction and assume that there exists $\bar x_i \in (l_i,u_i)$ such that $i \in A_l(\bar x) \cup A_u(\bar x)$.
From the definition of $A_l(\bar x)$ and $A_u(\bar x)$, it follows that $g_i(\bar x) \ne 0$, violating~\eqref{opt3} and thus contradicting the fact that $\bar x$ is a stationary
point.% This completes the proof.
\end{proof}

%%%%%%%%%%%%%%%%%%%%%%%%%%%%%%%%%%%%%%%%%
% active-set descent
%%%%%%%%%%%%%%%%%%%%%%%%%%%%%%%%%%%%%%%%%
\begin{proof}[Proof of Proposition~\ref{prop1}]
By the second-order mean value theorem, we have
\[
f(\tilde x) = f(x) + g(x)^T(\tilde x-x) + \frac{1}{2}(\tilde x-x)^T H(z)(\tilde x-x),
\]
where $z = x + \xi(\tilde x-x)$ for a $\xi \in ]0,1[$.
Therefore,
\begin{equation}\label{proof1:1}
f(\tilde x) - f(x) \le g(x)^T(\tilde x-x) + \frac{1}{2}\bar{\lambda}\|x-\tilde x\|^2. \\
\end{equation}
Recalling the definition of $\tilde x$, we can also write
\begin{equation}\label{proof1:2}
g(x)^T(\tilde x-x) = \sum_{i \in A_l(x)} g_i(x)(l_i-x_i) + \sum_{i \in A_u(x)} g_i(x)(u_i-x_i).
\end{equation}
From the definitions of $A_l(x)$ and $A_u(x)$, and recalling~\eqref{lambdafunc} and~\eqref{mufunc}, we have
\begin{align*}
g_i(x) \ge \frac{(x_i-l_i)}{\epsilon} \left[ \frac{(l_i-x_i)^2 + (u_i-x_i)^2}{(u_i-x_i)^2} \right], \quad & i \in A_l(x), \\
g_i(x) \le \frac{(x_i-u_i)}{\epsilon} \left[ \frac{(l_i-x_i)^2 + (u_i-x_i)^2}{(l_i-x_i)^2} \right], \quad & i \in A_u(x),
\end{align*}
and we can write
\begin{gather*}
g_i(x)(l_i-x_i) \le -\frac{1}{\epsilon}(x_i-l_i)^2\left[\frac{(l_i-x_i)^2 + (u_i-x_i)^2}{(u_i-x_i)^2}\right] \le -\frac{1}{\epsilon}(l_i-x_i)^2, \quad i \in A_l(x), \\
g_i(x)(u_i-x_i) \le -\frac{1}{\epsilon}(u_i-x_i)^2\left[\frac{(l_i-x_i)^2 + (u_i-x_i)^2}{(u_i-x_i)^2}\right] \le -\frac{1}{\epsilon}(u_i-x_i)^2, \quad i \in A_u(x).
\end{gather*}
Hence, from~\eqref{proof1:2}, it follows that
\begin{equation}\label{proof1:7}
g(x)^T(\tilde x-x) \le -\frac{1}{\epsilon}\left[\sum_{i\in A_l(x)}(l_i-x_i)^2 + \sum_{i\in A_u(x)}(u_i-x_i)^2\right] = -\frac{1}{\epsilon}\|x-\tilde x\|^2.
\end{equation}
Finally, from~\eqref{proof1:1} and \eqref{proof1:7},
we have
\[
f(\tilde x)-f(x) \le \frac{1}{2}\left(\bar{\lambda}-\frac{1}{\epsilon}\right)\|x-\tilde x\|^2 - \frac{1}{2\epsilon}\|x-\tilde x\|^2 \le -\frac{1}{2\epsilon}\|x-\tilde x\|^2,
\]
where the last inequality follows from equation~\eqref{eq_ass:eps} in Assumption~\ref{ass:eps}.
\end{proof}

%%%%%%%%%%%%%%%%%%%%%%%%%%%%%%%%%%%%%%%%%
% line search
%%%%%%%%%%%%%%%%%%%%%%%%%%%%%%%%%%%%%%%%%
\begin{proof}[Proof of Proposition~\ref{prop:line_search}]
Since the gradient is Lipschitz continuous over $[l,u]$, there exists $L<\infty$ such that for all $s\in [0,1]$ and for all $\alpha \ge 0$:
\[
\|g(\bar x)-g(\bar x-s[\bar x-\bar x(\alpha)])\| \le s L \|\bar x-\bar x(\alpha)\|, \quad \forall \bar x \in [l,u].
\]
By the mean value theorem, we have:
\begin{align}\label{dim:mean_value}
f(\bar x(\alpha))-f(\bar x) & = g(\bar x)^T(\bar x(\alpha) -x) + \int_0^1\big(g(\bar x-s[\bar x - \bar x(\alpha)])-g(\bar x)\big)^T \big(\bar x(\alpha) - \bar x\big)\ ds\nonumber \\
                  & \le g(\bar x)^T(\bar x(\alpha) - \bar x) +\|\bar x(\alpha) - \bar x\|\int_{0}^{1}s L \|\bar x(\alpha) - \bar x\|\ ds\nonumber \\
                  & = g(\bar x)^T(\bar x(\alpha) - \bar x) + \frac{L}{2}\|\bar x(\alpha) - \bar x\|^{2}, \quad \forall \alpha \ge 0.
\end{align}
Moreover, as the gradient is continuous and the feasible set is compact, there exists $M > 0$ such that
\begin{equation}\label{major_g}
\|g(\bar x)\| \le M, \quad \forall \bar x \in [l,u].
\end{equation}
From~\eqref{conddir1}, \eqref{conddir3} and~\eqref{major_g}, we can write
\[
d_i \le \|d\| \le \sigma_2\|g(\bar x)\| \le \sigma_2 M, \quad \forall \bar x \in [l,u], \quad \forall i=1,\dots,n.
\]
Now, let us define $\theta_1,\dots,\theta_n$ as:
\[
\theta_i :=
\begin{cases}
\min\ \{\bar x_i - l_i, u_i - \bar x_i\} \quad \text{ if } l_i < \bar x_i < u_i, \\
u_i - l_i \qquad \qquad \qquad \quad\,\, \text{ otherwise},
\end{cases} \qquad i=1,\dots,n.
\]
We set
\[
\tilde \theta := \min_{i=1,\dots,n} \frac{\theta_i}{2},
\]
and define $\hat\alpha$ as follows:
\[
\hat\alpha := \frac{\tilde\theta}{\sigma_2 M}.
\]
In the following, we want to majorize the right-hand-side term of~\eqref{dim:mean_value}. First, we consider the term $g(\bar x)^T(\bar x(\alpha) - x)$. We distinguish three cases:
\begin{description}
\item[(i)] $i \in N(\bar x)$ such that $l_i < \bar x_i < u_i$. We distinguish two subcases:
    \begin{itemize}
    \item if $d_i \ge 0$:
        \[
        l_i < \bar x_i + \alpha d_i \le \bar x_i + \frac{\tilde\theta}{\sigma_2 M} d_i
            \le x_i + \tilde\theta < u_i, \qquad \forall \alpha \in ]0,\hat\alpha],
        \]
    \item else, if $d_i < 0$:
        \[
        u_i > \bar x_i + \alpha d_i \ge \bar x_i + \frac{\tilde\theta}{\sigma_2 M} d_i
        \ge \bar x_i - \tilde \theta > l_i, \qquad \forall \alpha \in ]0,\hat\alpha].
        \]
    \end{itemize}
    So, we have
    \[
    \bar x_i(\alpha) = \bar x_i + \alpha d_i, \quad \forall \alpha \in ]0,\hat\alpha],
    \]
    which implies
    \begin{equation}\label{proof_box_major1}
    g_i(\bar x)(\bar x_i(\alpha) - \bar x_i) = \alpha g_i(\bar x)d_i, \quad \forall \alpha \in ]0,\hat\alpha].
    \end{equation}
\item[(ii)] $i \in N(\bar x)$ such that $\bar x_i = l_i$. Recalling the definition of $N(x)$, it follows that $g_i(\bar x) \le 0$. We distinguish two subcases:
    \begin{itemize}
    \item if $d_i \ge 0$:
        \[
        l_i \le \bar x_i + \alpha d_i \le \bar x_i + \frac{\tilde\theta}{\sigma_2 M} d_i
        \le \bar x_i + \tilde \theta < u_i, \qquad \forall \alpha \in ]0,\hat\alpha],
        \]
        and then
        \[
        \bar x_i(\alpha) = \bar x_i + \alpha d_i, \quad \forall \alpha \in ]0,\hat\alpha],
        \]
        which implies
        \begin{equation}\label{proof_box_major2}
        g_i(\bar x)(\bar x_i(\alpha) - \bar x_i) = \alpha g_i(\bar x)d_i, \quad \forall \alpha \in ]0,\hat\alpha].
        \end{equation}
    \item else, if $d_i < 0$, we have
        \[
        \bar x_i(\alpha) = \bar x_i, \quad \forall \alpha > 0,
        \]
        and then
        \begin{equation}\label{proof_box_major3}
        0 = g_i(\bar x)(\bar x_i(\alpha) - \bar x_i) \le \alpha g_i(\bar x)d_i, \quad \forall \alpha > 0.
        \end{equation}
    \end{itemize}
\item[(iii)] $i \in N(\bar x)$ such that $\bar x_i = u_i$.
Following the same reasonings done in the previous step, we have that
     \begin{itemize}
     \item if $d_i \le 0$:
         \begin{equation}\label{proof_box_major4}
         g_i(\bar x)(\bar x_i(\alpha) - \bar x_i) = \alpha g_i(\bar x)d_i, \quad \forall \alpha \in ]0,\hat\alpha];
         \end{equation}
     \item else, if $d_i > 0$, we have
        \begin{equation}\label{proof_box_major5}
        0 = g_i(\bar x)(\bar x_i(\alpha) - \bar x_i) \le \alpha g_i(\bar x)d_i, \quad \forall \alpha > 0.
        \end{equation}
    \end{itemize}
\end{description}
From~\eqref{conddir1}, \eqref{proof_box_major1}, \eqref{proof_box_major2}, \eqref{proof_box_major3}, \eqref{proof_box_major4} and~\eqref{proof_box_major5},
we obtain
\begin{equation}\label{proof_box_major6}
\begin{split}
g(\bar x)^T(\bar x(\alpha) - \bar x) & = \sum_{i \in N(\bar x)} g_i(\bar x)(\bar x_i(\alpha) - \bar x) \\
                                     & \le \alpha\sum_{i \in N(\bar x)} g_i(\bar x)d_i
                                       = \alpha g_{N(\bar x)}(\bar x)^Td_{N(\bar x)}, \quad \forall \alpha \in ]0,\hat\alpha].
\end{split}
\end{equation}
Now, we consider the term $\displaystyle\frac{L}{2}\|\bar x(\alpha) - \bar x\|^{2}$. For every $i \in N(\bar x)$ such that $d_i \le 0$,
we have that $0 \le \bar x_i - \bar x_i(\alpha)\le -\alpha d_i$ holds for all $\alpha>0$.
Therefore,
\begin{equation}\label{Nk1}
(\bar x_i - \bar x_i(\alpha))^2\le \alpha^2 d_i^2, \quad \forall \alpha > 0.
\end{equation}
Else, for every $i \in N(\bar x)$ such that $d_i > 0$, we have that
$0 \le \bar x_i(\alpha) - \bar x_i\le \alpha d_i$ holds for all $\alpha > 0$.
Therefore,
\begin{equation} \label{Nk2}
0 \le (\bar x_i(\alpha) - \bar x_i)^2\le \alpha^2 d_i^2, \quad \forall \alpha > 0.
\end{equation}
Recalling~\eqref{conddir1}, from~\eqref{Nk1} and~\eqref{Nk2} we obtain
\[
\|\bar x(\alpha) - \bar x\|^2 \le \alpha^2 \|d_{N(\bar x)}\|^2, \quad \forall \alpha > 0.
\]
Using~\eqref{conddir2} and~\eqref{conddir3}, we get
\begin{equation}\label{proof_box_major7}
\begin{split}
\|\bar x(\alpha) - \bar x\|^2 & \le \alpha^2 \|d_{N(\bar x)}\|^2 \le \alpha^2 \sigma_2^2\|g_{N(\bar x)}(\bar x)\|^2 \\
                              & \le -\alpha^2 \frac{\sigma_2^2}{\sigma_1} g_{N(\bar x)}(\bar x)^Td_{N(\bar x)}, \quad \forall \alpha > 0.
\end{split}
\end{equation}
From~\eqref{conddir1}, \eqref{dim:mean_value}, \eqref{proof_box_major6} and~\eqref{proof_box_major7}, we can write
\[
\begin{split}
f(\bar x(\alpha))-f(\bar x) & \le \alpha\left(1 - \alpha\frac{L\sigma^2_2}{2\sigma_1}\right)g_{N(\bar x)}(\bar x)^T d_{N(\bar x)} \\
                            & = \alpha\left(1 - \alpha\frac{L\sigma^2_2}{2\sigma_1}\right)g(\bar x)^T d, \quad \forall \alpha \in ]0,\hat\alpha].
\end{split}
\]
It follows that~\eqref{prop_decr} is satisfied by choosing $\bar \alpha$ such that
\begin{gather*}
1 - \bar \alpha\frac{L\sigma^2_2}{2\sigma_1} \ge \gamma, \\
\bar\alpha \in ]0,\hat\alpha].
\end{gather*}
Thus, the proof is completed defining
\[
\bar{\alpha} := \min\left\{\hat\alpha,\frac{2\sigma_1(1-\gamma)}{L\sigma^2_2}\right\}.
\]
\end{proof}

\section{}\label{app:alg}
\small \linespread{1}

The scheme of the algorithm is reported in Algorithm~\ref{alg:ASA_BCP}.
At Step~$10$,~$17$ and~$25$ there is the update of the reference value of the non-monotone line search $f^j_R$:
we set $j := j+1$, $l^j := k$ and the reference value is updated according to the formula
\[
f^j_R := \max_{0 \le i \le \min\{j,M\}} \bigl\{f^{j-i}\bigr\}.
\]

%\begin{breakablealgorithm}
\begin{algorithm}
\caption{\ASABCP}
\label{alg:ASA_BCP}
  \begin{algorithmic}
  \par\vspace*{0.1cm}
  \item$\,\,\,0$\hspace*{0.3truecm} Choose $x^0\in[l,u]$, fix $Z\ge 1$, $M\ge 0$, $\Delta_0\ge0$, $\beta\in ]0,1[$, $\delta\in ]0,1[$, $\gamma\in ]0,\frac 1 2[$, $k := 0$, \par
                                    $j := -1$, $l^{-1} := -1$, $f_R^{-1} := f(x^0)$, $f^{-1} := f(x^0)$, $\Delta = \tilde \Delta := \Delta_0$, $checkpoint := true$ \par
  \item$\,\,\,1$\hspace*{0.3truecm} While $x^k$ is a non-stationary point for problem~\eqref{prob}
  \item$\,\,\,2$\hspace*{0.8truecm} Compute $A_l^k:=A_l(x^k)$, $A_u^k:=A_u(x^k)$ and $N^k:=N(x^k)$
  \item$\,\,\,3$\hspace*{0.8truecm} Set $\tilde x^{k}_{A_l^k} := l_{A_l^k}$, $\tilde x^{k}_{A_u^k} := u_{A_u^k}$ and $\tilde x^{k}_{N^k} := x^k_{N^k}$
  \item$\,\,\,4$\hspace*{0.8truecm} If $\|\tilde x^k - x^k\| \le \tilde \Delta$, then set $\tilde \Delta = \beta \tilde \Delta$
  \item$\,\,\,5$\hspace*{0.8truecm} Else compute $f(x^k)$
  \item$\,\,\,6$\hspace*{1.3truecm} If $f(x^k) \ge f^j_R$, then backtrack to $\tilde x^{l^j}$, set $k := l^j$ and go to Step~$28$
  \item$\,\,\,7$\hspace*{0.8truecm} End if
  \item$\,\,\,8$\hspace*{0.8truecm} Compute $\tilde A_l^k:=A_l(\tilde x^k)$,  $\tilde A_u^k:=A_u(\tilde x^k)$ and $\tilde N^k:=N(\tilde x^k)$
  \item$\,\,\,9$\hspace*{0.8truecm} If $\tilde N^k \ne \emptyset$ and $g_{\tilde N^k}(\tilde x^k) \ne 0$
  \item$10$\hspace*{1.3truecm} If $checkpoint = true$, then compute $f(\tilde x^k)$ and update $f_R^j$
  \item$11$\hspace*{1.8truecm} Set $checkpoint := false$
  \item$12$\hspace*{1.3truecm} End if
  \item$13$\hspace*{1.3truecm} Set $d^{k}_{\tilde A_l^k} := 0$, $d^{k}_{\tilde A_u^k} := 0$ and compute a gradient-related direction $d^k_{\tilde N^k}$ in $\tilde x^k$
  \item$14$\hspace*{1.3truecm} If $k \ge l^j + Z$, then compute $f(\tilde x^k)$
  \item$15$\hspace*{1.8truecm} If $f(\tilde x^k) \ge f^j_R$
  \item$16$\hspace*{2.3truecm} Backtrack to $\tilde x^{l^j}$, set $d^k := d^{l^j}$, $k := l^j$ and go to Step~$28$
  \item$17$\hspace*{1.8truecm} Else update $f_R^j$
  \item$18$\hspace*{1.8truecm} End if
  \item$19$\hspace*{1.3truecm} End if
  \item$20$\hspace*{1.3truecm} If $\|{d^{k}_{\tilde N^k}}\| \le \Delta$
  \item$21$\hspace*{1.8truecm} Set $\alpha^k := 1$, $x^{k+1}:=[\tilde x^k + \alpha^k d^k]^{\sharp}$, $\Delta := \beta \Delta$, $k := k + 1$
  \item$22$\hspace*{1.3truecm} Else if $k \ne l^j$, then compute $f(\tilde x^k)$
  \item$23$\hspace*{1.8truecm} If $f(\tilde x^k) \ge f^j_R$
  \item$24$\hspace*{2.3truecm} Backtrack to $\tilde x^{l^j}$, set $d^k := d^{l^j}$, $k := l^j$ and go to Step~$28$
  \item$25$\hspace*{1.8truecm} Else update $f_R^j$
  \item$26$\hspace*{1.8truecm} End if
  \item$27$\hspace*{1.3truecm} End if
  \item$28$\hspace*{1.3truecm} Set $\alpha^k := \delta^{\nu}$, where $\nu$ is the smallest nonnegative integer for which
                               \[
                               f([\tilde x^k+\delta^{\nu} d^k]^{\sharp}) \le  f^j_R + \gamma \delta^{\nu} g(\tilde x^k)^T d^k
                               \]
  \item$29$\hspace*{1.3truecm} Set $x^{k+1} := [\tilde x^k+\alpha^k d^k]^{\sharp}$, $k := k + 1$, $checkpoint := true$
  \item$30$\hspace*{0.8truecm} Else
  \item$31$\hspace*{1.3truecm} Set $\alpha^k := 0$, $d^k := 0$, $x^{k+1}:=[\tilde x^k+\alpha^k d^k]^{\sharp}$, $k:=k+1$
  \item$32$\hspace*{0.8truecm} End if
  \item$33$\hspace*{0.3truecm} End while
  \par\vspace*{0.1cm}
  \end{algorithmic}
%\end{breakablealgorithm}
\end{algorithm}

\section{}\label{app:conv}
\small \linespread{1}

In this section, we prove Theorem~\ref{th:glob_conv}. Preliminarily, we need to state some results.

%%%%%%%%%%%%%%%%%%%%%%%%%%%%%%%%%%%%%%%%%
% nonmonotonocity
%%%%%%%%%%%%%%%%%%%%%%%%%%%%%%%%%%%%%%%%%
\begin{lemma}\label{lemma_nm1}
Let Assumption~\ref{ass:eps} hold. Suppose that \ASABCP\ produces an infinite sequence $\{x^k\}$, then
\begin{enumerate}[(i)]
\item $\{f^j_R\}$ is non-increasing and converges to a value $\bar f_R$;
\item for any fixed $j \ge 0$ we have:
\[
f^h_R < f^j_R, \quad \forall h > j+M.
\]
\end{enumerate}
\end{lemma}

%\begin{proof} From the instructions of the algorithm we have that, for every $j \ge 0$, $f^j_R= \displaystyle \max_{0 \le i \le m^j} f^{j-i}$.
%Since $m^{j+1} \le m^j+1$ we can write
%\begin{align*}
%f^{j+1}_R & = \displaystyle \max_{0 \le i \le m^{j+1}} f(\tilde x^{l^{(j+1-i)}}) \le \max_{0 \leq i \le m^j+1} f(\tilde x^{l^{(j+1-i)}}) \\
%          & = \displaystyle \max \left\{ f(\tilde x^{l^{(j+1[}}) , \max_{0 \le i \le m^j} f^{j-i}  \right\} = \max \left\{ f(\tilde x^{l^{(j+1)}}) , f^j_R \right\}.
%\end{align*}
%As $f(\tilde x^{l^{(j+1)}}) < f^j_R$, the above relation implies that
%\[\label{relaz1}
%f^{j+1}_R \le f^j_R \le f^0_R \le f(x^0),
%\]
%which proves that $\{f^j_R\}$ is nonincreasing. Since $\tilde x^{l^j}$ is feasible, the nonincreasing sequence $\{f^j_R\}$ is bounded from below, hence
%\[
%\lim_{j \rightarrow \infty} f^j_R=\bar f_R,
%\]
%which proves (i).
%\noindent
%Point (ii) follows from the relation $ f(\tilde x^{l^{(h+1)}}) < f^h_R$, and from the fact that in the algorithm $f^h_R$ is calculated considering at most $M+1$ values of the objective function.
%\end{proof}

\begin{proof}
The proof follows from Lemma~1 in~\cite{grippo:1991}.
\end{proof}

\begin{lemma}\label{lemma_nm2}
Let Assumption~\ref{ass:eps} hold. Suppose that \ASABCP\ produces an infinite sequence $\{x^k\}$ and an infinite sequence $\{\tilde x^k\}$.
For any given value of $k$, let $q(k)$ be the index such that
\[
q(k) := \max \{ j \colon l^j \le k \}.
\]
Then, there exists a sequence $\{\tilde x^{s(j)}\}$ and an integer $L$ satisfying the following conditions:
\begin{enumerate}[(i)]
\item $f^j_R = f(\tilde x^{s(j)})$
\item for any integer $k$, there exist an index $h^k$ and an index $j^k$ such that:
\begin{gather*}
0 < h^k - k \le L, \qquad h^k = s(j^k), \\
f^{j^k}_R = f(\tilde x^{h^k}) < f^{q(k)}_R.
\end{gather*}
\end{enumerate}
\end{lemma}

\begin{proof}
The proof follows from Lemma~$2$ in~\cite{grippo:1991} taking into account that for any iteration index $k$, there exists an integer $\tilde L$ such that the condition of Step~$9$ is satisfied within the $(k+\tilde L)$-th iteration. In fact, assume by contradiction that it is not true. If Step~$9$ is not satisfied at a generic iteration $k$, then $x^{k+1} = \tilde x^k$. Since the sequences $\{x^k\}$ and $\{\tilde x^k\}$ are infinite, Proposition~\ref{prop:stationary_point3} implies that $\tilde x^{k+1} \ne x^{k+1}$ and that the objective function strictly decreases. Repeating this procedure for an infinite number of steps, an infinite sequence of distinct points $\{x^{k+1}, x^{k+2}, \dots\}$ is produced, where these points differ from each other only for the values of the variables at the bounds. Since the number of variables is finite, this produces a contradiction.
\end{proof}

\begin{lemma}\label{lemma_nm3}
Let Assumption~\ref{ass:eps} hold. Suppose that \ASABCP\ produces an infinite sequence $\{x^k\}$ and an infinite sequence $\{\tilde x^k\}$. Then,
\begin{gather}
\label{conv_f_R} \lim_{k \to \infty} f(x^{k}) = \lim_{k \to \infty} f(\tilde x^{k}) = \lim_{j \rightarrow \infty} f^j_R = \bar f_R, \\
\label{conv_x} \lim_{k \to \infty} \| x^{k+1} - \tilde x^k \| = \lim_{k \to \infty} \alpha^k \|d^k\| = 0, \\
\label{conv_x_tilde} \lim_{k \to \infty} \| \tilde x^k - x^k \| = 0.
\end{gather}
\end{lemma}

\begin{proof}
We build two different partitions of the iterations indices to analyze the computation of $x^{k+1}$ from $\tilde x^k$ and that of $\tilde x^k$ from $x^k$, respectively.
From the instructions of the algorithm, it follows that $x^{k+1}$ can be computed at Step~$21$, Step~$29$ or Step~$31$. Let us consider the following subset of iteration indices:
\begin{gather*}
\mathcal K_1 := \{k \colon x^{k+1} \text{ is computed at Step $21$}\}, \\
\mathcal K_2 := \{k \colon x^{k+1} \text{ is computed at Step $29$}\}, \\
\mathcal K_3 := \{k \colon x^{k+1} \text{ is computed at Step $31$}\}.
\end{gather*}
Then, we have
\[
\mathcal K_1 \cup \mathcal K_2 \cup \mathcal K_3 = \{0,1,\dots\}.
\]
As regards the computation of $\tilde x^k$, we distinguish two further subsets of iterations indices:
\begin{gather*}
\mathcal K_4 := \{k \colon \tilde x^k \text{ satisfies the test at Step $4$}\}, \\
\mathcal K_5 := \{k \colon \tilde x^k \text{ does not satisfy the test at Step $4$}\}.
\end{gather*}
Then, we have
\[
\mathcal K_4 \cup \mathcal K_5 = \{0,1,\dots\}.
\]
Preliminarily, we point out some properties of the above subsequences. The subsequence $\{\tilde x^k\}_{\mathcal K_1}$ satisfies
\[
\| x^{k+1} - \tilde x^k \| = \alpha^k\|d^k\| = \|d^k\| \le \beta^t\Delta_0, \quad k \in \mathcal K_1,
\]
where the integer $t$ increases with $k \in \mathcal K_1$. Since $\beta \in ]0,1)$, if $\mathcal K_1$ is infinite, we have
\begin{equation}\label{conv_k1}
\lim_{k \to \infty, \, k \in \mathcal K_1} \| x^{k+1} - \tilde x^k \| = \lim_{k \to \infty, \, k \in \mathcal K_1} \alpha^k \| d^k \| = 0.
\end{equation}
Moreover, since $\alpha^k = 0$ and $d^k = 0$ for all $k \in \mathcal K_3$, if $\mathcal K_3$ is infinite, we have
\begin{equation}\label{conv_k3}
\lim_{k \to \infty, \, k \in \mathcal K_3} \| x^{k+1} - \tilde x^k \| = \lim_{k \to \infty, \, k \in \mathcal K_3} \alpha^k \| d^k \| = 0.
\end{equation}
The subsequence $\{\tilde x^k\}_{\mathcal K_4}$ satisfies
\[
\|\tilde x^k - x^k\| \le \beta^t\tilde \Delta_0, \quad k \in \mathcal K_4,
\]
where the integer $t$ increases with $k \in \mathcal K_4$. Since $\beta \in ]0,1[$, if $\mathcal K_4$ is infinite, we have
\begin{equation}\label{conv_k4}
\lim_{k \to \infty, \, k \in \mathcal K_4} \|\tilde x^k - x^k\| = 0.
\end{equation}
Now we prove~\eqref{conv_f_R}. Let $s(j)$, $h^k$ and $q(k)$ be the indices defined in Lemma~\ref{lemma_nm2}. We show that for any fixed integer $i \ge 1$, the following relations hold:
\begin{gather}
\label{conv_nm1} \lim_{j \to \infty} \| \tilde x^{s(j)-i+1} - x^{s(j)-i+1} \| = 0, \\
\label{conv_nm2} \lim_{j \to \infty} \| x^{s(j)-i+1} - \tilde x^{s(j)-i} \| = \lim_{j \to \infty} \alpha^{s(j)-i} \| d^{s(j)-i} \| = 0, \\
\label{conv_nm3} \lim_{j \to \infty} f(x^{s(j)-i+1}) = \bar f_R, \\
\label{conv_nm4} \lim_{j \to \infty} f(\tilde x^{s(j)-i}) = \bar f_R.
\end{gather}
Without loss of generality, we assume that $j$ is large enough to avoid the occurrence of negative apices. We proceed by induction and first show that~\eqref{conv_nm1}--\eqref{conv_nm4} hold for $i=1$.
If $s(j) \in \mathcal K_4$, relations~\eqref{conv_nm1} and~\eqref{conv_nm3} follow from~\eqref{conv_k4} and the continuity of the objective function. If $s(j) \in \mathcal K_5$, from the instructions of the algorithm and taking into account Proposition~\ref{prop1}, we get
\[
f^j_R = f(\tilde x^{s(j)}) \le f(x^{s(j)}) - \dfrac{1}{2\epsilon}\|x^{s(j)} - \tilde x^{s(j)}\|^2 < f^{j-1}_R,
\]
from which we get
\[
f^j_R = f(\tilde x^{s(j)}) \le f(x^{s(j)}) < f^{j-1}_R,
\]
and then, from point~(i) of Lemma~\ref{lemma_nm1}, it follows that
\begin{equation}\label{f_x_s(j)_to_f_x_tilde_s(j)}
\lim_{j \to \infty} f(\tilde x^{s(j)}) = \lim_{j \to \infty} f(x^{s(j)}) = \bar f_R,
\end{equation}
which proves~\eqref{conv_nm3} for $i=1$. From the above relation, and by exploiting Proposition~\ref{prop1} again, we have that
\[
\lim_{j \to \infty} \bigl( f(\tilde x^{s(j)}) - f(x^{s(j)}) \bigr) \le \lim_{j \to \infty} - \dfrac{1}{2\epsilon}\|x^{s(j)} - \tilde x^{s(j)}\|^2.
\]
and then~\eqref{conv_nm1} holds for $i=1$.\\
If $s(j)-1 \in \mathcal K_1 \cup \mathcal K_3$, from~\eqref{conv_k1} and~\eqref{conv_k3} it is straightforward to verify that~\eqref{conv_nm2} holds for $i=1$. By exploiting the continuity of the objective function, since~\eqref{conv_nm2} and~\eqref{conv_nm3} hold for $i=1$, then also~\eqref{conv_nm4} is verified for $i=1$.
If $s(j)-1 \in \mathcal {K}_2$, from the instruction of the algorithm, we obtain
\[
f(x^{s(j)}) = f(\tilde x^{s(j)-1} + \alpha^{s(j)-1} d^{s(j)-1}) \le f^{q(s(j)-1)}_R + \gamma \alpha^{s(j)-1} g(\tilde x^{s(j)-1})^T d^{s(j)-1},
\]
and then
\[
f(x^{s(j)}) - f^{q(s(j)-1)}_R \le \gamma \alpha^{s(j)-1} g(\tilde x^{s(j)-1})^T d^{s(j)-1}.
\]
From~\eqref{f_x_s(j)_to_f_x_tilde_s(j)}, point~(i) of Lemma~\ref{lemma_nm1}, and recalling~\eqref{conddir1}--\eqref{conddir3}, we have that
\[
\lim_{j \to \infty} \alpha^{s(j)-1} \|d^{s(j)-1}\| = \lim_{j \to \infty} \| x^{s(j)} - \tilde x^{s(j)-1} \| = 0
\]
for every subsequence such that $s(j)-1 \in \mathcal K_2$. Therefore,~\eqref{conv_nm2} holds for $i=1$.
Recalling that $f(x^{s(j)}) = f(\tilde x^{s(j)-1} + \alpha^{s(j)-1} d^{s(j)-1})$, and since~\eqref{conv_nm1} and~\eqref{conv_nm2} hold for $i=1$, from the continuity of the objective function it follows that also~\eqref{conv_nm4} holds for $i=1$.\\
Now we assume that~\eqref{conv_nm1}--\eqref{conv_nm4} hold for a given fixed $i \ge 1$ and show that these relations must hold for $i + 1$ as well.
If $s(j) - i \in \mathcal K_4$, by using~\eqref{conv_k4}, it is straightforward to verify that~\eqref{conv_nm1} is verified replacing $i$ with $i+1$.
Taking into account~\eqref{conv_nm4}, this implies that
\[
\lim_{j \to \infty} f(x^{s(j) - (i+1) +1}) = \lim_{j \to \infty} f(x^{s(j) - i}) = \lim_{j \to \infty} f(\tilde x^{s(j)-i}) = \bar f_R,
\]
and then~\eqref{conv_nm3} holds for $i+1$. If $s(j) - i \in \mathcal K_5$, from the instructions of the algorithm, and taking into account Proposition~\ref{prop1}, we get
\[
f(\tilde x^{s(j)-i}) \le f(x^{s(j)-i}) - \dfrac{1}{2\epsilon}\|x^{s(j)-i} - \tilde x^{s(j)-i}\|^2 < f^{q(s(j)-i)-1}_R.
\]
Exploiting~\eqref{conv_nm4} and point~(i) of Lemma~\ref{lemma_nm1}, we have that
\begin{equation}\label{f_x_s(j)_to_f_x_tilde_s(j)_2}
\lim_{j \to \infty} f(\tilde x^{s(j)-i}) = \lim_{j \to \infty} f(x^{s(j)-i}) = \bar f_R,
\end{equation}
which proves~\eqref{conv_nm3} for $i+1$. From the above relation, and by exploiting Proposition~\ref{prop1} again, we can also write
\[
\lim_{j \to \infty} \bigl( f(\tilde x^{s(j)-i}) - f(x^{s(j)}-i) \bigr) \le \lim_{j \to \infty} - \dfrac{1}{2\epsilon}\|x^{s(j)-i} - \tilde x^{s(j)-i}\|^2.
\]
and then~\eqref{conv_nm1} holds for $i+1$.\\
If $s(j)-i-1 \in \mathcal K_1 \cup \mathcal K_3$, from~\eqref{conv_k1} and~\eqref{conv_k3}, we obtain that~\eqref{conv_nm2} holds for $i=i+1$.
Since $x^{s(j)-i} = \tilde x^{s(j)-i-1} + \alpha^{s(j)-i-1} d^{s(j)-i-1}$, exploiting~\eqref{conv_k1}, \eqref{conv_k3}, \eqref{f_x_s(j)_to_f_x_tilde_s(j)_2} and the continuity of the objective function, we obtain that~\eqref{conv_nm4}
holds replacing $i$ with $i+1$.\\
If $s(j)-i-1 \in \mathcal {K}_2$, from the instruction of the algorithm, we obtain
\[
\begin{split}
f(x^{s(j)-i}) & = f(\tilde x^{s(j)-i-1} + \alpha^{s(j)-i-1} d^{s(j)-i-1}) \\
              & \le f^{q(s(j)-i-1)}_R + \gamma \alpha^{s(j)-i-1} g(\tilde x^{s(j)-i-1})^T d^{s(j)-i-1},
\end{split}
\]
and then
\[
f(x^{s(j)-i}) - f^{q(s(j)-i-1)}_R \le \gamma \alpha^{s(j)-i-1} g(\tilde x^{s(j)-i-1})^T d^{s(j)-i-1}.
\]
From~\eqref{f_x_s(j)_to_f_x_tilde_s(j)_2}, point~(i) of Lemma~\ref{lemma_nm1}, and recalling~\eqref{conddir1}--\eqref{conddir3}, we have that
\[
\lim_{j \to \infty} \alpha^{s(j)-i-1} \|d^{s(j)-i-1}\| = \lim_{j \to \infty} \| x^{s(j)-i} - \tilde x^{s(j)-i-1} \| = 0
\]
for every subsequence such that $s(j)-1 \in \mathcal K_2$. Therefore,~\eqref{conv_nm2} holds for $i+1$.\\
Recalling that $f(x^{s(j)-i}) = f(\tilde x^{s(j)-i-1} + \alpha^{s(j)-i-1} d^{s(j)-i-1})$, and since~\eqref{conv_nm1} and~\eqref{conv_nm2} hold replacing $i$ with $i+1$, exploiting the continuity of the objective function, we have
\[
\lim_{j \to \infty} f(\tilde x^{s(j)-i-1} + \alpha^{s(j)-i-1} d^{s(j)-i-1}) = \lim_{j \to \infty} f(x^{s(j)-i}) = \lim_{j \to \infty} f(\tilde x^{s(j)-i}).
\]
Therefore, if~\eqref{conv_nm4} holds at a generic $i \ge 1$, it must hold for $i+1$ as well. This completes the induction.\\
Now, for any iteration index $k>0$, we can write
\[
\begin{split}
& \tilde x^{h^k} = x^k + \sum_{i=0}^{h^k - k} \bigl(\tilde x^{h^k-i} - x^{h^k-i}\bigr) + \sum_{i=1}^{h^k - k} \alpha^{h^k-i} d^{h^k-i}, \\
& \tilde x^{h^k} = \tilde x^k + \sum_{i=0}^{h^k - k - 1} \bigl(\tilde x^{h^k-i} - x^{h^k-i}\bigr) + \sum_{i=1}^{h^k - k} \alpha^{h^k-i} d^{h^k-i}.
\end{split}
\]
From~\eqref{conv_nm1} and~\eqref{conv_nm2}, we obtain
\[
\lim_{k \to \infty} \|x^k - \tilde x^{h^k}\| = \lim_{k \to \infty} \|\tilde x^k - \tilde x^{h^k}\| = 0.
\]
By exploiting the continuity of the objective function, and taking into account point~(i) of Lemma~\ref{lemma_nm1}, the above relation implies that
\[
\lim_{k \to \infty} f(x^k) = \lim_{k \to \infty} f(\tilde x^k) = \lim_{k \to \infty} f(\tilde x^{h^k}) = \lim_{k \to \infty} f^j_R = \bar f_R,
\]
which proves~\eqref{conv_f_R}.\\
To prove~\eqref{conv_x}, if $k \in \mathcal K_1 \cup \mathcal K_3$, then from~\eqref{conv_k1} and~\eqref{conv_k3} we obtain
\begin{equation}\label{conv_x_k1_k3}
\lim_{k \to \infty, \, k \in \mathcal K_1 \cup \mathcal K_3} \| x^{k+1} - \tilde x^k \| = \lim_{k \to \infty, \, k \in \mathcal K_1 \cup \mathcal K_3} \alpha^k \| d^k \| = 0.
\end{equation}
If $k \in \mathcal K_2$, from the instruction of the algorithm, we get
\[
f(x^{k+1}) \le f(\tilde x^{q(k)}) + \gamma \alpha^k g(\tilde x^k)^T d^k,
\]
and then, recalling conditions~\eqref{conddir1}--\eqref{conddir3} and~\eqref{conv_f_R}, we can write
\begin{equation}\label{conv_x_k2}
\lim_{k \to \infty, \, k \in \mathcal K_2} \| x^{k+1} - \tilde x^k \| = \lim_{k \to \infty, \, k \in \mathcal K_2} \alpha^k \|d^k\| = 0.
\end{equation}
From~\eqref{conv_x_k1_k3} and~\eqref{conv_x_k2}, it follows that~\eqref{conv_x} holds.\\
To prove~\eqref{conv_x_tilde}, if $k \in \mathcal K_4$, then from~\eqref{conv_k4} we obtain
\begin{equation}\label{conv_x_tilde_k4}
\lim_{k \to \infty, \, k \in \mathcal K_4} \| \tilde x^k - x^k \| = 0.
\end{equation}
If $k \in \mathcal K_5$, from the instruction of the algorithm and recalling Proposition~\ref{prop1}, we get
\[
f(\tilde x^k) \le f(x^k) - \dfrac{1}{2\epsilon}\|x^k - \tilde x^k\|^2 < f^{q(k)-1}_R.
\]
From~\eqref{conv_f_R} and point~(i) of Lemma~\ref{lemma_nm1}, we have that
\[
\lim_{k \to \infty, \, k \in \mathcal K_5} f(\tilde x^k) = \lim_{k \to \infty, \, k \in \mathcal K_5} f(x^k) = \bar f_R.
\]
By exploiting Proposition~\ref{prop1} again, we can write
\[
\lim_{k \to \infty, \, k \in \mathcal K_5} \bigl( f(\tilde x^k) - f(x^k) \bigr) \le \lim_{k \to \infty, \, k \in \mathcal K_5} - \dfrac{1}{2\epsilon}\|x^k - \tilde x^k\|^2 = 0,
\]
and then
\begin{equation}\label{conv_x_tilde_k5}
\lim_{k \to \infty, \, k \in \mathcal K_5} \| \tilde x^k - x^k \| = 0.
\end{equation}
From~\eqref{conv_x_tilde_k4} and~\eqref{conv_x_tilde_k5}, it follows that~\eqref{conv_x_tilde} holds.% This completes the proof.
\end{proof}

%%%%%%%%%%%%%%%%%%%%%%%%%%%%%%%%%%%%%%%%%
% directional derivatives
%%%%%%%%%%%%%%%%%%%%%%%%%%%%%%%%%%%%%%%%%
The following theorem extends a known result from unconstrained optimization, guaranteeing that the sequence of the directional derivatives
along the search direction converges to zero.
\begin{theorem}\label{th:lim_dir_der}
Let Assumption~\ref{ass:eps} hold. Assume that \ASABCP\ does not terminate in a finite number of iterations, and let $\{x^k\}$, $\{\tilde x^k\}$ and $\{d^k\}$ be the sequences produced by the algorithm. Then,
\begin{equation}\label{lim_gd}
\lim_{k\to\infty} g(\tilde x^k)^Td^k = 0.
\end{equation}
\end{theorem}

\begin{proof}
We can identify two iteration index subsets $H,K \subseteq \{1,2,\dots\}$, such that:
\begin{itemize}
\item $N(\tilde x^k) \neq \emptyset$ and $g_N(\tilde x^k) \neq 0$, for all $k \in K$,
\item $H := \{1,2,\dots\} \setminus K$.
\end{itemize}
By assumption, the algorithm does not terminate in a finite number of iterations,
and then, at least one of the above sets is infinite.
Since we are interested in the asymptotic behavior of the sequence produced by \ASABCP,
we assume without loss of generality that both $H$ and $K$ are infinite sets. \\
Taking into account Step~$31$ in Algorithm~1, it is straightforward to verify that
\[
\lim_{k\to\infty, \, k\in H} g(\tilde x^k)^Td^k = 0.
\]
Therefore, we limit our analysis to consider the subsequence $\{x^k\}_K$. Let $\bar x$ be any limit point of $\{x^k\}_K$.
By contradiction, we assume that~\eqref{lim_gd} does not hold.
Using~\eqref{conv_x_tilde} of Lemma~\ref{lemma_nm3}, since $\{x^k\}$, $\{\tilde x^k\}$ and $\{d^k\}$ are limited,
and taking into account that $A_l(x^k)$, $A_u(x^k)$ and $N(x^k)$ are subsets of a finite set of indices,
without loss of generality we redefine $\{x^k\}_K $ the subsequence such that
\[
\lim_{k\to\infty,\, k\in K}x^k = \lim_{k\to\infty,\, k\in K}\tilde x^k = \bar x,
\]
and
\begin{gather*}
N^k := \hat N,\quad A_l^k := \hat A_l,\quad A_u^k := \hat A_u,\quad \forall k \in K, \\
\lim_{k\to\infty, \, k\in K} d^k = \hat d. \label{lim_d}
\end{gather*}
Since we have assumed that~\eqref{lim_gd} does not hold, the above relations, combined with~\eqref{conddir2} and the continuity of the gradient,
imply that
\begin{equation}\label{lim_gd_contr}
\lim_{k\to\infty, \, k\in K} g(\tilde x^k)^Td^k = g(\bar x)^T \hat d = -\eta < 0.
\end{equation}
It follows that
\[
\lim_{k\to\infty, \, k\in K} \hat d \ne 0,
\]
and then, recalling~\eqref{conv_x} of Lemma~\ref{lemma_nm3}, we get
\begin{equation}\label{contr_alpha}
\lim_{k\to\infty, \, k\in K} \alpha^k = 0.
\end{equation}
Consequently, from the instructions of the algorithm, there must exist a subsequence (renamed $K$ again) such that the line search procedure at Step~$28$
is performed and $\alpha^k < 1$ for sufficiently large $k$. Namely,
\begin{equation}\label{arm_not_satisf1}
\begin{split}
f\bigl(\bigl[\tilde x^k + \frac{\alpha^k}{\delta}d^k\bigr]^{\sharp}\bigr) & > f^{q(k)}_R + \gamma \frac{\alpha^k}{\delta} g(\tilde x^k)^Td^k \\
& \ge f(\tilde x^k) + \gamma \frac{\alpha^k}{\delta} g(\tilde x^k)^Td^k, \quad \forall k \ge \bar k, \, k \in K,
\end{split}
\end{equation}
where $q(k) := \max \{ j \colon l^j \le k \}$.
We can write the point $[\tilde x^k + \frac{\alpha^k}{\delta}d^k]^{\sharp}$ as follows:
\begin{equation}\label{x_armijo_rej}
\bigl[\tilde x^k + \frac{\alpha^k}{\delta}d^k\bigl]^{\sharp} = \tilde x^k + \frac{\alpha^k}{\delta}d^k - y^k,
\end{equation}
where
\[
y^k_i := \max\bigl\{0, \bigl(\tilde x^k + \frac{\alpha^k}{\delta}d^k\bigr)_i - u_i\bigr\} - \max\bigl\{0,l_i-\bigl(\tilde x^k + \frac{\alpha^k}{\delta}d^k\bigr)_i\bigr\}, \quad i=1,\dots,n.
\]
As $\{\tilde x^k\}$ is a sequence of feasible points, $\{\alpha^k\}$ converges to zero and $\{d^k\}$ is limited, we get
\begin{equation}\label{lim_y_k}
\lim_{k\to\infty, \, k\in K} y^k = 0.
\end{equation}
From~\eqref{arm_not_satisf1} and~\eqref{x_armijo_rej}, we can write
\begin{equation}\label{arm_not_satisf2}
f\bigl(\tilde x^k + \frac{\alpha^k}{\delta}d^k - y^k\bigr) - f(\tilde x^k) > \gamma \frac{\alpha^k}{\delta} g(\tilde x^k)^Td^k, \quad \forall k \ge \bar k, \, k \in K.
\end{equation}
By the mean value theorem, we have
\begin{equation}\label{mean_th_point}
f\bigl(\tilde x^k + \frac{\alpha^k}{\delta}d^k - y^k\bigr) = f(\tilde x^k) + \frac{\alpha^k}{\delta}g(z^k)^Td^k - g(z^k)^Ty^k,
\end{equation}
where
\begin{equation}\label{z_k}
z^k = \tilde x^k + \theta^k\bigl(\frac{\alpha^k}{\delta}d^k - y^k\bigr), \quad \theta^k \in ]0,1[.
\end{equation}
From~\eqref{contr_alpha} and~\eqref{lim_y_k}, and since $\{d^k\}$ is limited, we obtain
\begin{equation}\label{lim_z}
\lim_{k\to\infty, \, k\in K} z^k = \bar x.
\end{equation}
Substituting~\eqref{mean_th_point} into~\eqref{arm_not_satisf2}, and multiplying each term by $\dfrac{\delta}{\alpha^k}$, we get
\begin{equation}\label{arm_not_satisf3}
g(z^k)^Td^k - \frac{\delta}{\alpha^k}g(z^k)^Ty^k > \gamma g(\tilde x^k)^Td^k, \quad \forall k \ge \bar k, \, k \in K.
\end{equation}
From the definition of $y^k$, it follows that
\begin{equation}\label{y_k}
y^k_i =
\begin{cases}
0, \qquad & \text{ if } l_i \le \tilde x^k_i + \frac{\alpha^k}{\delta}d^k_i \le u_i, \\
\tilde x^k_i + \frac{\alpha^k}{\delta}d^k_i - u_i > 0, \qquad & \text{ if } \tilde x^k_i + \frac{\alpha^k}{\delta}d^k_i > u_i, \\
\tilde x^k_i + \frac{\alpha^k}{\delta}d^k_i - l_i < 0, \qquad & \text{ if } l_i > \tilde x^k_i + \frac{\alpha^k}{\delta}d^k_i.
\end{cases}
\end{equation}
In particular, we have
\begin{equation}\label{y_k_sign}
y^k_i
\begin{cases}
= 0, \quad & \text{ if } d^k_i = 0, \\
\in [0 , \frac{\alpha^k}{\delta} d^k_i], \quad & \text{ if } d^k_i > 0, \\
\in [-\frac{\alpha^k}{\delta} d^k_i, 0], \quad & \text{ if } d^k_i < 0.
\end{cases}
\end{equation}
From the above relation, it is straightforward to verify that
\begin{equation}\label{y_vs_d}
|y^k_i| \le \frac{\alpha^k}{\delta}|d^k_i|, \quad i=1,\dots,n.
\end{equation}
In the following, we want to majorize the left-hand side of~\eqref{arm_not_satisf3} by showing that $\{\frac{\alpha^k}{\delta}g(z^k)^Ty^k\}$ converges to a nonnegative value. To this aim, we analyze three different cases, depending on whether $\bar x_i$ is at the bounds or is strictly feasible:
\begin{description}
\item[(i)] $i \in \hat N$ such that $l_i < \bar x_i < u_i$. As $\{\tilde x^k\}$ converges to $\bar x$, there exists $\tau > 0$ such that
    \[
    l_i + \tau \le \tilde x^k \le u_i - \tau, \qquad k \in K, \, k \text{ sufficiently large}.
    \]
    Since $\{\alpha^k\}$ converges to zero and $\{d^k\}$ is limited, it follows that $\frac{\alpha^k}{\delta} |d^k_i| < \tau$, for $k \in K$, $k$~sufficiently large. Then,
    \[
    l_i < \tilde x^k_i + \frac{\alpha^k}{\delta}d^k_i < u_i, \qquad k \in K, \, k \text{ sufficiently large},
    \]
    which implies, from~\eqref{y_k}, that
    \begin{equation}\label{y_zero_1}
    y^k_i = 0, \qquad k \in K, \, k \text{ sufficiently large}.
    \end{equation}
\item[(ii)] $i \in \hat N$ such that $\bar x_i = l_i$. First, we show that
    \begin{align}
    & g_i(\bar x) \le 0, \label{g_i_x} \\
    & y^k_i \le 0, \qquad k \in K, \, k \text{ sufficiently large}. \label{lim_y}
    \end{align}
    To show~\eqref{g_i_x}, we assume by contradiction that $g_i(\bar x) > 0$. From~\eqref{lambdafunc} and recalling that $\|\tilde x^k - x^k\|$ converges to zero from~\eqref{conv_x_tilde} of Lemma~\ref{lemma_nm3}, it follows that
    \[
    \lim_{k\to\infty, \, k\in K} \lambda_i(\tilde x^k) = \lim_{k\to\infty, \, k\in K} g_i(\tilde x^k) = g_i(\bar x) > 0.
    \]
    Then, there exist an iteration index $\hat k$ and a scalar $\xi > 0$ such that $\lambda_i(\tilde x^k) \ge \xi > 0$, for all $k \ge \hat k, \, k \in K$.
    As $\{\tilde x^k_i\}$ converges to $l_i$, there also exists $\tilde k \ge \hat k$ such that
    \begin{align*}
    l_i \le \tilde x^k_i \le l_i + \epsilon \xi \le l_i + \epsilon \lambda_i(\tilde x^k), & \qquad k \in K, \, k \ge \tilde k,  \\
    g_i(\tilde x^k) > 0, & \qquad k \in K, \, k \ge \tilde k,
    \end{align*}
    which contradicts the fact that $i \in N(\tilde x^k)$ for $k$ sufficiently large.
    To show~\eqref{lim_y}, we observe that since $\{\tilde x^k_i\}$ converges to $l_i$, there exists $\tau \in ]0,u_i-l_i]$ such that
    \[
    l_i \le \tilde x^k_i \le u_i-\tau, \qquad k \in K, \, k \text{ sufficiently large}.
    \]
    Moreover, since $\{\alpha^k\}$ converges to zero and $\{d^k\}$ is limited, it follows that $\frac{\alpha^k}{\delta} d^k_i \le \tau$, for $k \in K$, $k$~sufficiently large. Then,
    \[
    \tilde x^k + \frac{\alpha^k}{\delta} d^k_i \le u_i, \qquad k \in K, \, k \text{ sufficiently large}.
    \]
    The above relation, combined with~\eqref{y_k}, proves~\eqref{lim_y}.
    Now, we distinguish two subcases, depending on the sign of $d^k_i$:
    \begin{itemize}
    \item for every subsequence $\bar K \subseteq K$ such that $d^k_i \ge 0$, from~\eqref{y_k_sign} it follows that $y^k_i \geq 0$. Consequently, from~\eqref{lim_y} we can write
        \begin{equation}\label{y_zero_2}
        y^k_i = 0, \qquad k \in \bar K, \, k \text{ sufficiently large}.
        \end{equation}
    \item for every subsequence $\bar K \subseteq K$ such that $d^k_i < 0$, we have two further possible situations, according to~\eqref{g_i_x}:
        \begin{itemize}
        \item[(a)] $g_i(\bar x) < 0$. As $\{z^k\}$ converges to $\bar x$, then $g_i(z^k) \le 0$ for $k \in \bar K$, $k$ sufficiently large. From~\eqref{lim_y}, we obtain
            \begin{equation}\label{y_zero_3}
            \frac{\delta}{\alpha^k}g_i(z^k)y^k_i \ge 0, \qquad k \in \bar K, \, k \text{ sufficiently large}.
            \end{equation}
        \item[(b)] $g_i(\bar x) = 0$. From~\eqref{y_vs_d}, we get
            \[
            \frac{\delta}{\alpha^k} |g_i(z^k)y^k_i| \le \frac{\delta}{\alpha^k} |g_i(z^k)| |y^k_i| \le |g_i(z^k)| |d^k_i|.
            \]
            Since $\{d^k\}$ is limited, $\{z^k\}$ converges to $\bar x$, and $g_i(\bar x) = 0$, from the continuity of the gradient we get
            \begin{equation}\label{y_zero_4}
            \lim_{k\to\infty, \, k\in \bar K} \frac{\delta}{\alpha^k} g_i(z^k)d^k_i = 0.
            \end{equation}
        \end{itemize}
    \end{itemize}
\item[(iii)] $i \in \hat N$ such that $\bar x_i = u_i$. Reasoning as in the previous case, we obtain
    \begin{equation}\label{y_zero_6}
    \lim_{k\to\infty, \, k\in K} \frac{\delta}{\alpha^k} g_i(z^k)y^k_i \ge 0.
    \end{equation}
\end{description}
Finally, from~\eqref{y_zero_1}, \eqref{y_zero_2}, \eqref{y_zero_3}, \eqref{y_zero_4} and~\eqref{y_zero_6}, we have
\begin{equation}\label{lim_delta_alpha_g_y}
\lim_{k\to\infty, \, k\in K} \frac{\delta}{\alpha^k} g(z^k)^Ty^k \ge 0,
\end{equation}
and, from~\eqref{lim_gd_contr}, \eqref{lim_z}, \eqref{arm_not_satisf3}, \eqref{y_zero_6} and~\eqref{lim_delta_alpha_g_y}, we obtain
\[
\begin{split}
 -\eta & = \lim_{\substack{k\to\infty \\ k\in K}} g(\tilde x^k)^T d^k = \lim_{\substack{k\to\infty \\ k\in K}} g(z^k)^T d^k \ge \lim_{\substack{k\to\infty \\ k\in K}} g(z^k)^T d^k - \lim_{\substack{k\to\infty \\ k\in K}} \frac{\delta}{\alpha^k} g(z^k)^Ty^k \\
& \ge \lim_{\substack{k\to\infty \\ k\in K}} \gamma g(\tilde x^k)^Td^k = -\gamma \eta.
\end{split}
\]
This contradicts the fact that we set $\gamma < 1$ in \ASABCP.
\end{proof}

Now, we can prove Theorem~\ref{th:glob_conv}.

%%%%%%%%%%%%%%%%%%%%%%%%%%%%%%%%%%%%%%%%%
% global convergence
%%%%%%%%%%%%%%%%%%%%%%%%%%%%%%%%%%%%%%%%%
\begin{proof}[Proof of Theorem~\ref{th:glob_conv}]
Let $x^*$ be any limit point of the sequence $\{x^k\}$, and let $\{x^k\}_K $ be the subsequence converging to $x^*$.
From~\eqref{conv_x_tilde} of Lemma~\ref{lemma_nm3} we can write
\begin{equation}\label{lim_x_tilde2}
\lim_{k\to\infty, \, k\in K}\tilde x^{k}=x^*,
\end{equation}
and, thanks to the fact that $A_l(x^k)$, $A_u(x^k)$ and $N(x^k)$ are subsets of a finite set of indices, we can define a further subsequence $\hat K \subseteq K$ such that
\[
N^k := \hat N, \quad A_l^k := \hat A_l, \quad A_u^k := \hat A_u,
\]
for all $k \in \hat{K}$.
Recalling Proposition~\ref{prop:stationary_point}, we define the following function that measures the violation of the optimality conditions for feasible points:
\[
\phi(x_i) = \min \left\{ \max\{l_i-x_i, -g_i(x)\}^2, \max\{x_i-u_i, g_i(x)\}^2 \right\}.
\]
By contradiction, we assume that $x^*$ is a non-stationary point for problem~\eqref{prob}. Then, there exists an index $i$ such that $\phi(x^*_i) > 0$. From~\eqref{lim_x_tilde2} and the continuity of $\phi$, there exists an index $\tilde{k}$
such that
\begin{equation}\label{violation}
\phi(\tilde{x}^k_i) \ge \Delta > 0, \quad \forall k \ge \tilde k.
\end{equation}
Now, we consider three cases:
\begin{description}
\item[(i)] $i \in \hat A_l$. Then, $\tilde x^k_i = l_i$. From~\eqref{lambdafunc} and \eqref{Al}, we get $g_i(x^k) > 0, \, \forall k \in \hat K$. By continuity of the gradient, and since both $\{\tilde{x}^k\}_{\hat K}$ and $\{x^k\}_{\hat K}$ converge to $x^*$, we obtain
    \[
    g_i(\tilde x^k) \ge -\frac{\Delta}{2},
    \]
    for $k \in \hat{K}$, $k$ sufficiently large. Then, we have $\phi(\tilde{x}^k_i) \le \frac{\Delta^2}{4} < \Delta$ for $k \in \hat{K}$, $k$ sufficiently large. This contradicts~\eqref{violation}.
\item[(ii)] $i \in \hat A_u$. Then, $\tilde x^k_i = u_i$. The proof of this case is a verbatim repetition of the previous case.
\item[(iii)] $i \in \hat N$. As $\phi(x_i^*) > 0$, then $g_i(x^*) \ne 0$. From Theorem~\ref{th:lim_dir_der}, we have
    \[
    \lim_{k\to\infty,\, k\in \hat K} g(\tilde x^k)^T d^k = 0.
    \]
    From~\eqref{conddir2}, it follows that
    \[
    \lim_{k\to\infty,\, k\in \hat K} \|g_{\hat N}(\tilde x^k)\| = \|g_{\hat N}(x^*)\| = 0,
    \]
    leading to a contradiction. %This completes the proof.
\end{description}
\end{proof}

%%%%%%%%%%%%%%%%%%%%%%%%%%%%%%%%%%%%%%%%%
% superlinearity
%%%%%%%%%%%%%%%%%%%%%%%%%%%%%%%%%%%%%%%%%
In order to prove Theorem~\ref{th:superlin}, we need a further lemma.
\begin{lemma}\label{lemma:conv1}
Let Assumption~\ref{ass:eps} hold and assume that $\{x^k\}$ is an infinite sequence generated by \ASABCP. Then, there exists an iteration index $\bar k$ such that
$N(x^k) \neq \emptyset$ for all $k \geq \bar k$.
\end{lemma}

\begin{proof}
By contradiction, we assume that there exists an infinite index subset $\bar K \subseteq \{1,2,\dots\}$ such that $N(x^k) = \emptyset$ for all $k \in \bar K$.
Let $x^*$ be a limit point of $\{x\}_{\bar K}$, that is,
\[
\lim_{k \to \infty, \, k \in K} x^k = x^*,
\]
where $K \subseteq \bar K$. Theorem~\ref{th:glob_conv} ensures that $x^*$ is a stationary point. From~\eqref{conv_x_tilde} of Lemma~\ref{lemma_nm3}, we can write
\[
\lim_{k \to \infty, \, k \in K} x^k = \lim_{k \to \infty, \, k \in K} \tilde x^k = x^*.
\]
Moreover, from Proposition~\ref{prop:estim}, there exists an index $\hat k$ such that
\begin{gather}
\{i \colon x^*_i = l_i, \lambda^*_i > 0 \} \subseteq A_l(x^k) \subseteq \{i \colon x^*_i = l_i\}, \quad \forall \ k \geq \hat k, \quad k \in K, \label{A_l_included} \\
\{i \colon x^*_i = u_i, \mu^*_i > 0 \} \subseteq A_u(x^k) \subseteq \{i \colon x^*_i = u_i\}, \quad \forall \ k \geq \hat k, \quad k \in K. \label{A_u_included}
\end{gather}
Let $\tilde k$ be the smallest integer such that $\tilde k \geq \hat k$ and $\tilde k \in K$. From~\eqref{A_l_included} and~\eqref{A_u_included}, we can write
\begin{align*}
\tilde x^{\tilde k}_i = l_i = x^*_i, \quad & \text{ if } i \in A_l(x^{\tilde k}), \\
\tilde x^{\tilde k}_i = u_i = x^*_i, \quad & \text{ if } i \in A_u(x^{\tilde k}).
\end{align*}
Since $N(x^k)$ is empty for all $k \in K$, we also have
\[
A_l(x^k) \cup A_u(x^k) = \{1,\dots,n\}, \quad \forall \ k \in K.
\]
Consequently, $\tilde x^{\tilde k} = x^*$, contradicting the hypothesis that the sequence $\{x^k\}$ is infinite. %This completes the proof.
\end{proof}

Now, we can finally prove Theorem~\ref{th:superlin}.

\begin{proof}[Proof of Theorem~\ref{th:superlin}]
From Proposition~\ref{prop:estim}, exploiting the fact the sequence $\{x^k\}$ converges to $x^*$ and that strict complementarity holds,
we have that for sufficiently large $k$,
\begin{align*}
& N(x^k) = N(\tilde x^k) = N^*, \\
& A_l(x^k) = A_l(\tilde x^k) = \{i \colon x^*_i = l_i\}, \\
& A_u(x^k) = A_u(\tilde x^k) = \{i \colon x^*_i = u_i\}.
\end{align*}
From the instructions of the algorithm, it follows that $\tilde x^k = x^k$ for sufficiently large $k$,
and then, the minimization is restricted on $N(\tilde x^k)$.
From Lemma~\ref{lemma:conv1}, we have that $N(\tilde x^k) = N(x^k) = N^* \ne \emptyset$ for sufficiently large $k$.
Furthermore, from~\eqref{tronc}, we have that $d_{N(\tilde x^k)}^k$ is a Newton-truncated direction, and then,
the assertion follows from standard results on unconstrained minimization.
\end{proof}
\end{appendices}


\begin{thebibliography}{10}
\providecommand{\url}[1]{{#1}}
\providecommand{\urlprefix}{URL }
\expandafter\ifx\csname urlstyle\endcsname\relax
  \providecommand{\doi}[1]{DOI~\discretionary{}{}{}#1}\else
  \providecommand{\doi}{DOI~\discretionary{}{}{}\begingroup
  \urlstyle{rm}\Url}\fi

\bibitem{conn:1988}
Conn, A.R., Gould, N.I., Toint, P.L.: {Global convergence of a class of trust
  region algorithms for optimization with simple bounds}.
\newblock SIAM J. Numer. Anal. \textbf{25}(2), 433--460 (1988)

\bibitem{lin:1999}
Lin, C.J., Mor{\'e}, J.J.: {Newton's method for large bound-constrained
  optimization problems}.
\newblock SIAM J. Optim. \textbf{9}(4), 1100--1127 (1999)

\bibitem{dennis:1998}
Dennis, J., Heinkenschloss, M., Vicente, L.N.: {Trust-region interior-point SQP
  algorithms for a class of nonlinear programming problems}.
\newblock SIAM J. Control Optim. \textbf{36}(5), 1750--1794 (1998)

\bibitem{heinkenschloss:1999}
Heinkenschloss, M., Ulbrich, M., Ulbrich, S.: {Superlinear and quadratic
  convergence of affine-scaling interior-point Newton methods for problems with
  simple bounds without strict complementarity assumption}.
\newblock Math. Program. \textbf{86}(3), 615--635 (1999)

\bibitem{kanzow:2006}
Kanzow, C., Klug, A.: {On affine-scaling interior-point Newton methods for
  nonlinear minimization with bound constraints}.
\newblock Comput. Optim. Appl. \textbf{35}(2), 177--197 (2006)

\bibitem{bertsekas:1982}
Bertsekas, D.P.: {Projected Newton methods for optimization problems with
  simple constraints}.
\newblock SIAM J. Control Optim. \textbf{20}(2), 221--246 (1982)

\bibitem{facchinei:2002}
Facchinei, F., Lucidi, S., Palagi, L.: {A truncated Newton algorithm for large
  scale box constrained optimization}.
\newblock SIAM J. Optim. \textbf{12}(4), 1100--1125 (2002)

\bibitem{hager:2006}
Hager, W.W., Zhang, H.: {A new active set algorithm for box constrained
  optimization}.
\newblock SIAM J. Optim. \textbf{17}(2), 526--557 (2006)

\bibitem{schwartz:1997}
Schwartz, A., Polak, E.: {Family of projected descent methods for optimization
  problems with simple bounds}.
\newblock J. Optim. Theory Appl. \textbf{92}(1), 1--31 (1997)

\bibitem{facchinei:1998}
Facchinei, F., J{\'u}dice, J., Soares, J.: {An active set Newton algorithm for
  large-scale nonlinear programs with box constraints}.
\newblock SIAM J. Optim. \textbf{8}(1), 158--186 (1998)

\bibitem{cheng:2012}
Cheng, W., Li, D.: {An active set modified Polak--Ribiere--Polyak method for
  large-scale nonlinear bound constrained optimization}.
\newblock J. Optim. Theory Appl. \textbf{155}(3), 1084--1094 (2012)

\bibitem{andreani:2010}
Andreani, R., Birgin, E.G., Mart{\'\i}nez, J.M., Schuverdt, M.L.: {Second-order
  negative-curvature methods for box-constrained and general constrained
  optimization}.
\newblock Comput. Optim. Appl. \textbf{45}(2), 209--236 (2010)

\bibitem{birgin:2002}
Birgin, E.G., Mart{\'\i}nez, J.M.: {Large-scale active-set box-constrained
  optimization method with spectral projected gradients}.
\newblock Comput. Optim. Appl. \textbf{23}(1), 101--125 (2002)

\bibitem{desantis:2012}
De~Santis, M., Di~Pillo, G., Lucidi, S.: {An active set feasible method for
  large-scale minimization problems with bound constraints}.
\newblock Comput. Optim. Appl. \textbf{53}(2), 395--423 (2012)

\bibitem{facchinei:1995}
Facchinei, F., Lucidi, S.: {Quadratically and superlinearly convergent
  algorithms for the solution of inequality constrained minimization problems}.
\newblock J. Optim. Theory Appl. \textbf{85}(2), 265--289 (1995)

\bibitem{barzilai:1988}
Barzilai, J., Borwein, J.M.: {Two-point step size gradient methods}.
\newblock IMA J. Numer. Anal. \textbf{8}(1), 141--148 (1988)

\bibitem{desantis:2016}
De~Santis, M., Lucidi, S., Rinaldi, F.: {A Fast Active Set Block Coordinate
  Descent Algorithm for $\ell_1$-regularized least squares}.
\newblock SIAM J. Optim. \textbf{26}(1), 781--809 (2016)

\bibitem{buchheim:2015}
Buchheim, C., De~Santis, M., Lucidi, S., Rinaldi, F., Trieu, L.: {A feasible
  active set method with reoptimization for convex quadratic mixed-integer
  programming}.
\newblock SIAM J. Optim., \textbf{26}(3), 1695–1714 (2016)

\bibitem{dipillo:1984}
Di~Pillo, G., Grippo, L.: {A class of continuously differentiable exact penalty
  function algorithms for nonlinear programming problems}.
\newblock In: System Modelling and Optimization, pp. 246--256. Springer, Berlin (1984)

\bibitem{grippo2:1991}
Grippo, L., Lucidi, S.: {A differentiable exact penalty function for bound
  constrained quadratic programming problems}.
\newblock Optimization \textbf{22}(4), 557--578 (1991)

\bibitem{zhang:2004}
Zhang, H., Hager, W.W.: {A nonmonotone line search technique and its
  application to unconstrained optimization}.
\newblock SIAM J. Optim. \textbf{14}(4), 1043--1056 (2004)

\bibitem{dembo:1983}
Dembo, R.S., Steihaug, T.: {Truncated-Newton algorithms for large-scale
  unconstrained optimization}.
\newblock Math. Program. \textbf{26}(2), 190--212 (1983)

\bibitem{grippo:1991}
Grippo, L., Lampariello, F., Lucidi, S.: {A class of nonmonotone stabilization
  methods in unconstrained optimization}.
\newblock Numer. Math. \textbf{59}(1), 779--805 (1991)

\bibitem{gould:2003}
Gould, N.I., Orban, D., Toint, P.L.: {GALAHAD, a library of thread-safe Fortran
  90 packages for large-scale nonlinear optimization}.
\newblock ACM Trans. Math. Softw. (TOMS) \textbf{29}(4), 353--372 (2003)

\bibitem{gould:2015}
Gould, N.I., Orban, D., Toint, P.L.: {CUTEst: a constrained and unconstrained
  testing environment with safe threads for mathematical optimization}.
\newblock Comput. Optim. Appl. \textbf{60}(3), 545--557 (2015)

\bibitem{dolan:2002}
Dolan, E.D., Mor{\'e}, J.J.: {Benchmarking optimization software with
  performance profiles}.
\newblock Math. Program. \textbf{91}(2), 201--213 (2002)

\bibitem{birgin:2012}
Birgin, E.G., Gentil, J.M.: {Evaluating bound-constrained minimization
  software}.
\newblock Comput. Optim. Appl. \textbf{53}(2), 347--373 (2012)

\end{thebibliography}
\end{document}